\documentclass[10pt,oneside,leqno]{amsart}
\usepackage{amsxtra}
\usepackage{amsopn}
\usepackage{color}
\usepackage{amsmath,amsthm,amssymb}
\usepackage{amscd}
\usepackage{amsfonts}
\usepackage{latexsym}
\usepackage{verbatim}

\theoremstyle{plain}
\newtheorem{theorem}{Theorem}[section]
\newtheorem{definition}[theorem]{Definition}
\newtheorem{lemma}[theorem]{Lemma}
\newtheorem{proposition}[theorem]{Proposition}
\newtheorem{corollary}[theorem]{Corollary}
\newtheorem{remark}[theorem]{Remark}

\newtheorem{example}[theorem]{Example}
\newtheorem{question}[theorem]{Question}
\newtheorem{remark-question}[section]{Remark-Question}
\newtheorem{conjecture}[section]{Conjecture}

\newcommand\C{{\mathbb C}}

\newcommand\R{{\mathbb R}}

\newcommand\frg{{\mathfrak g}}
\newcommand\frh{{\mathfrak h}}

\newcommand\gc{\frg_\mathbb{C}}
\newcommand\Real{{\mathfrak R}{\frak e}\,} 
\newcommand\Imag{{\mathfrak I}{\frak m}\,}
\newcommand\nilm{\Gamma\backslash G}

\newcommand\db{{\bar{\partial}}}
\newcommand\zzz{{\!\!\!}}

\begin{document}
\title[]{Invariant complex structures on 6-nilmanifolds: classification, Fr\"olicher spectral sequence and special Hermitian metrics}


\author{M. Ceballos}
\address[M. Ceballos]{Departamento de Geometr\'{\i}a y Topolog\'{\i}a\\
 Universidad de Sevilla\\
Fa\-cul\-tad de Matem\'aticas. Apartado 1160\\
41080 Sevilla, Spain}
\email{mceballos@us.es}

\author{A. Otal}
\address[A. Otal and R. Villacampa]{Centro Universitario de la Defensa\,-\,I.U.M.A., Academia General
Mili\-tar, Crta. de Huesca s/n. 50090 Zaragoza, Spain}
\email{aotal@unizar.es}
\email{raquelvg@unizar.es}

\author{L. Ugarte}
\address[L. Ugarte]{Departamento de Matem\'aticas\,-\,I.U.M.A.\\
Universidad de Zaragoza\\
Campus Plaza San Francisco\\
50009 Zaragoza, Spain}
\email{ugarte@unizar.es}

\author{R. Villacampa}


\maketitle

\begin{abstract}
We classify invariant complex structures on 6-dimensional nilmanifolds up to equivalence.
As an application, the behaviour of the associated Fr\"olicher sequence is studied
as well as its relation to the
existence of strongly Gauduchon metrics.
We also show that the strongly Gauduchon property and the balanced property are not closed under
holomorphic deformation.
\end{abstract}


\section{Introduction}\label{intro}

\noindent
Let $\frg$ be a Lie algebra endowed with an endomorphism $J\colon \frg\longrightarrow \frg$
such that $J^2=-{\rm Id}$. The endomorphism $J$ is a \emph{complex structure}
if the integrability condition
$$
[JX,JY]=J[JX,Y]+J[X,JY]+[X,Y]
$$
is satisfied for any $X,Y\in \frg$;
equivalently, the $i$-eigenspace $\frg_{1,0}$ of $J$ in $\gc=\frg\otimes_{\mathbb{R}}\mathbb{C}$ is a complex subalgebra of $\gc$.
Nilpotent Lie algebras $\frg$ admitting a complex structure were classified by Salamon \cite{S} up to dimension 6.
More recently, Andrada, Barberis and Dotti classified in \cite{ABD} the 6-dimensional Lie algebras $\frg$ having
a complex structure $J$ of abelian type, that is, the complex subalgebra $\frg_{1,0}$ is abelian, or equivalently
$[JX,JY]=[X,Y]$ for any $X,Y\in\frg$.

A related question is to determine the complex structures on a given Lie algebra $\frg$ up
to isomorphism in the following sense. Two complex structures $J$
and $J'$ on $\frg$ are \emph{equivalent} if there exists an automorphism
$F\colon \frg\longrightarrow\frg$ of the Lie algebra such that
$J=F^{-1}\circ J'\circ F$. The latter condition is equivalent to say that~$F$,
extended to $\gc$, satisfies $F(\frg_{1,0}^J) \subset \frg_{1,0}^{J'}$.
If $\mathcal{C}(\frg)$ denotes the space of complex structures on $\frg$ then $\mathcal{C}(\frg)/{\rm Aut}(\frg)$ parametrizes the equivalence classes
of complex structures on $\frg$.

A classification of abelian complex structures in dimension 6 is given in \cite{ABD}.
Some partial results on nilpotent Lie algebras can be found in several papers \cite{CP,KS,U,UV1},
although to our knowledge there is no complete classification of complex structures on 6-dimensional
nilpotent Lie algebras. This is our first goal here.

The classification of complex structures on
nilpotent Lie algebras provides a classification of \emph{invariant} complex structures on
nilmanifolds. Let $M=\nilm$ be a nilmanifold, i.e. a compact quotient of a
simply-connected nilpotent Lie group $G$ by a lattice $\Gamma$ of
maximal rank. If $J$ is a complex structure on the Lie algebra $\frg$ of $G$,
then it gives rise to a left-invariant complex structure on $G$ which descends
to a complex structure on the quotient $M$ in a natural way.
Several interesting aspects of this complex geometry have been investigated, as for instance
the Dolbeault cohomology~\cite{CF,CFGU2,R1}, complex deformations~\cite{CP,CFP,MPPS,R2} or
the existence of special Hermitian metrics~\cite{FPS,U}. Recently, it is proved in~\cite{BDV} that the
canonical bundle of any complex nilmanifold is holomorphically trivial
and some applications to hypercomplex geometry are given.

As a first application of the classification of complex structures
we study the behaviour of the Fr\"olicher sequence~\cite{Fro}.
Recall that the Fr\"olicher sequence $E_r(M, J)$ of a complex manifold $(M, J)$ is the
spectral sequence associated to the double complex $(\Omega^{p,q}(M,J),\partial,\db)$, where
$\partial+\db=d$ is the decomposition, with respect to $J$, of the exterior differential $d$.
The first term $E_1(M,J)$ is precisely the Dolbeault cohomology of~$(M,J)$ and after
a finite number of steps the sequence converges to
the de Rham cohomology of $M$.
The first examples of compact complex manifolds for which $E_2\not\cong E_{\infty}$ were independently
found in~\cite{CFG1} and~\cite{Pittie}. The examples in~\cite{CFG1} are complex nilmanifolds of complex dimension~3,
which is the lowest possible dimension for which the Fr\"olicher
sequence can be non-degenerate at $E_2$.
More recently, Rollenske has constructed in~\cite{R0} complex nilmanifolds for which the sequence $\{E_r\}$
can be arbitrarily non-degenerate. The behaviour of the Fr\"olicher sequence has been studied for some other
complex manifolds~\cite{ES,Tanre}, but as far as we know its general behaviour
for complex nilmanifolds has not been studied, although some partial results can be found in~\cite{CFG2,CFGU0,CFGU1bis}.
Here we study the Fr\"olicher spectral sequence for general invariant complex structures on a 6-dimensional nilmanifold.
A remarkable consequence of this study is the existence of a compact complex manifold on which the $\partial\db$-lemma fails but
$E_1\cong E_{\infty}$ and the Hodge diamond is symmetric.

As a second application of the classification of complex structures
we consider strongly Gauduchon (sG for short) metrics in the sense of Popovici~\cite{Pop0,Pop1}.
Any balanced Hermitian metric is sG and any sG metric is a Gauduchon metric~\cite{Gau}.
In~\cite{Pop2} the relation between the degeneration of the Fr\"olicher sequence at $E_1$ and the existence
of sG metrics is studied, showing that these two notions are unrelated.
We study the existence of sG or balanced metrics on 6-nilmanifolds in relation to the general
behaviour of the Fr\"olicher sequence.
Moreover, Popovici proved in \cite{Pop1} that the sG property of compact complex manifolds is open under
holomorphic deformations, and conjectured in~\cite{Pop2} that the sG property and the balanced property
of compact complex manifolds are closed under holomorphic deformations.
We construct a counterexample to both closedness conjectures.

The paper is structured as follows. In Section~\ref{ComplexStructures} we first review some general facts about complex structures
on a 6-dimensional nilpotent Lie algebra $\frg$. By \cite{S} such $\frg$ must be isomorphic to $\frh_1,\ldots,\frh_{16}$,
$\frh_{19}^-$ or $\frh_{26}^+$ (see Theorem~\ref{clasif-complex} for a description of the Lie algebras).
Of special interest is $\frh_5$ because it corresponds to the real Lie algebra underlying the Iwasawa manifold,
whose complex geometry is studied in~\cite{KS}.
For the first sixteen classes the complex structure is necessarily of \emph{nilpotent} type
in the sense of \cite{CFGU2}.
We classify the non-abelian nilpotent complex structures on 2-step and 3-step nilpotent Lie algebras in Sections~\ref{2step}
and~\ref{3step}, respectively. Then, using the classification of non-nilpotent complex structures obtained in~\cite{UV1} as well as
the classification of abelian structures given in~\cite{ABD}, we present in Tables 1 and~2 of
Section~\ref{ClassifGeneral} the complete classification of complex structures
on 6-dimensional nilpotent Lie algebras up to equivalence.

Since $J$ equivalent to $J'$ implies that the terms in the associated Fr\"olicher sequences are isomorphic, as an application
we study the general behaviour of the Fr\"olicher sequence $E_r(\nilm,J)$ in
Section~\ref{Frolicher} (see Theorem~\ref{Fro-general} for details).
We find that $E_2\not\cong E_{\infty}$ if and only if the underlying Lie algebra $\frg\cong \frh_{13},\frh_{14}$ or $\frh_{15}$.
Moreover, $E_1\cong E_2\not\cong E_3 \cong E_{\infty}$ for any $J$ when $\frg\cong \frh_{13}$ or $\frh_{14}$.
In contrast, $\frh_{15}$ has a rich complex geometry with respect to Fr\"olicher sequence
because it admits complex structures for which
$E_1\not\cong E_2\cong E_{\infty}$,
$E_1\cong E_2\not\cong E_3 \cong E_{\infty}$
or even $E_1\not\cong E_2\not\cong E_3 \cong E_{\infty}$.
In Example~\ref{example-h15} we give a
family $J_t$ of non-equivalent complex structures on $\frh_{15}$ along which the Fr\"olicher sequence
has these three behaviours.
We also show that a nilmanifold with underlying Lie algebra $\frh_6$ has a complex structure with degenerate Fr\"olicher sequence and
satisfying $h_{\db}^{p,q}=h_{\db}^{q,p}$ for every $p,q\in \mathbb{N}$,
which provides an answer to a question recently posed in \cite{AT}
(see Proposition~\ref{h6-question}).

In section~\ref{sG} we study the existence of sG metrics on 6-dimensional nilmanifolds endowed with an invariant complex structure
and show that the underlying Lie algebra must be isomorphic to $\frh_1,\ldots,\frh_6$ or $\frh_{19}^-$. It is also proved that
the existence of sG metric implies the degeneration of the Fr\"olicher sequence at $E_2$.
Using~\cite{UV2} we give in Proposition~\ref{examples-sG-nobalanced} a classification of complex structures
having sG metrics but not admitting any balanced metric, as well as a
classification of nilpotent complex structures admitting balanced metric (see Table~3).
Based on the complex geometry of the Lie algebra $\frh_4$, in Theorem~\ref{counterexample} we show that neither the sG property
nor the balanced property of compact complex manifolds are closed under holomorphic deformation.

\section{Nilpotent complex structures on 6-dimensional\\ nilpotent Lie algebras}\label{ComplexStructures}

\noindent
Given a Lie algebra $\frg$, let $\gc^*$ be the dual of the complexification $\gc$ of $\frg$.
If $J\colon \frg \longrightarrow \frg$ is an endomorphism such that
$J^2=-{\rm Id}$, then there is a natural bigraduation induced on $\bigwedge^* \,\gc^* =\oplus_{p,q}
\bigwedge^{p,q}(\frg^*)$, where the spaces $\bigwedge^{1,0}(\frg^*)$
and $\bigwedge^{0,1}(\frg^*)$, which we shall also denote by
$\frg^{1,0}$ and $\frg^{0,1}$, are the eigenspaces of the
eigenvalues $\pm i$ of $J$ as an endomorphism of~$\gc^*$,
respectively.
Now, if $d\colon \bigwedge^* \gc^* \longrightarrow \bigwedge^{*+1}
\gc^*$ is the extension to the complexified exterior algebra of the
usual Chevalley-Eilenberg differential, then it is well known that $J$ is
a complex structure if and only if $\pi_{0,2} \circ d\vert_{\frg^{1,0}}
\equiv 0$, where $\pi_{0,2}\colon \bigwedge^{2} \gc^*
\longrightarrow \bigwedge^{0,2}(\frg^*)$ denotes the canonical projection.

We shall focus on \emph{nilpotent} Lie algebras (NLA for short).
Salamon has proved in~\cite{S} the following equivalent condition for
the integrability of $J$ on a $2n$-dimensional NLA $\frg$: $J$ is a complex
structure on $\frg$ if and only if $\frg^{1,0}$ has a basis
$\{\omega^j\}_{j=1}^n$ such that $d\omega^1=0$ and
$$
d \omega^{j} \in \mathcal{I}(\omega^1,\ldots,\omega^{j-1}), \quad
\mbox{ for } j=2,\ldots,n ,
$$
where $\mathcal{I}(\omega^1,\ldots,\omega^{j-1})$ is the ideal in
$\bigwedge\phantom{\!}^* \,\gc^*$ generated by
$\{\omega^1,\ldots,\omega^{j-1}\}$.

Recall that a complex structure $J$ on a $2n$-dimensional NLA $\frg$ is
\emph{nilpotent}~\cite{CFGU2} if there exists a basis $\{\omega^j\}_{j=1}^n$
for~$\frg^{1,0}$ satisfying $d\omega^1=0$ and
\begin{equation}\label{nilpotent-condition}
d \omega^j \in \bigwedge\phantom{\!\!}^2 \,\langle
\omega^1,\ldots,\omega^{j-1},
\omega^{\overline{1}},\ldots,\omega^{\overline{j-1}} \rangle, \quad
\mbox{ for } j=2,\ldots,n .
\end{equation}
An important special class of nilpotent complex structures is the
\emph{abelian} class consisting of those structures $J$ satisfying $[JX,JY]=[X,Y]$, for all $X,Y\in
\frg$, or equivalently $d(\frg^{1,0})\subset \bigwedge^{1,1}(\frg^*)$. They are also characterized by the fact that
the subalgebra $\frg^{1,0}$ is abelian.

In six dimensions, the classification of NLAs in
terms of the different types of complex structures that they admit
is as follows.

\begin{theorem}\label{clasif-complex}\cite{S,U}
Let $\frg$ be an NLA of dimension~$6$. Then, $\frg$ has a complex
structure if and only if it is isomorphic to one of the following
Lie algebras:
$$
\begin{array}{rcl}
\frh_{1} &\!\!=\!\!& (0,0,0,0,0,0),\\[-2pt]
\frh_{2} &\!\!=\!\!& (0,0,0,0,12,34),\\[-2pt]
\frh_{3} &\!\!=\!\!& (0,0,0,0,0,12+34),\\[-2pt]
\frh_{4} &\!\!=\!\!& (0,0,0,0,12,14+23),\\[-2pt]
\frh_{5} &\!\!=\!\!& (0,0,0,0,13+42,14+23),\\[-2pt]
\frh_{6} &\!\!=\!\!& (0,0,0,0,12,13),\\[-2pt]
\frh_{7} &\!\!=\!\!& (0,0,0,12,13,23),\\[-2pt]
\frh_{8} &\!\!=\!\!& (0,0,0,0,0,12),\\[-2pt]
\frh_{9} &\!\!=\!\!& (0,0,0,0,12,14+25),
\end{array}
\quad\quad\quad
\begin{array}{rcl}
\frh_{10} &\!\!=\!\!& (0,0,0,12,13,14),\\[-2pt]
\frh_{11} &\!\!=\!\!& (0,0,0,12,13,14+23),\\[-2pt]
\frh_{12} &\!\!=\!\!& (0,0,0,12,13,24),\\[-2pt]
\frh_{13} &\!\!=\!\!& (0,0,0,12,13+14,24),\\[-2pt]
\frh_{14} &\!\!=\!\!& (0,0,0,12,14,13+42),\\[-2pt]
\frh_{15} &\!\!=\!\!& (0,0,0,12,13+42,14+23),\\[-2pt]
\frh_{16} &\!\!=\!\!& (0,0,0,12,14,24),\\[-2pt]
\frh^-_{19} &\!\!=\!\!& (0,0,0,12,23,14-35),\\[-2pt]
\frh^+_{26} &\!\!=\!\!& (0,0,12,13,23,14+25).
\end{array}
$$
Moreover:
\begin{enumerate}
\item[{\rm ({\it a})}] Any complex structure on $\frh_{19}^-$ and
$\frh_{26}^+$ is non-nilpotent;
\item[{\rm ({\it b})}] For $1\leq k\leq 16$, any complex structure
on $\frh_{k}$ is nilpotent;
\item[{\rm ({\it c})}] Any complex structure on $\frh_1$,
$\frh_{3}$, $\frh_{8}$ and $\frh_{9}$ is abelian;
\item[{\rm ({\it d})}] There exist both abelian and non-abelian
nilpotent complex structures on $\frh_{2}$, $\frh_{4}$, $\frh_{5}$
and $\frh_{15}$;
\item[{\rm ({\it e})}] Any complex structure on $\frh_{6}$,
$\frh_{7}$, $\frh_{10}$, $\frh_{11}$, $\frh_{12}$, $\frh_{13}$,
$\frh_{14}$ and $\frh_{16}$ is not abelian.
\end{enumerate}
\end{theorem}

\begin{remark}\label{remmm}
{\rm Here we use the usual notation, i.e. for instance
$\frh_{2}=(0,0,0,0,12,34)$ means that there is a basis
$\{e^j\}_{j=1}^6$ satisfying
$d e^1=d e^2=d e^3=d e^4=0$,
$d e^5= e^1\wedge e^2$,
$d e^6= e^3\wedge e^4$; equivalently, the Lie bracket is
given in terms of its dual basis $\{e_j\}_{j=1}^6$ by
$[e_1,e_2]=-e_5$, $[e_3,e_4]=-e_6$.}
\end{remark}

Let $\frg$ be a Lie algebra endowed with two complex structures $J$
and $J'$. We recall that $J$ and $J'$ are said to be {\it
equivalent} if there is an automorphism $F\colon \frg\longrightarrow
\frg$ of the Lie algebra such that $J'=F^{-1}\circ J\circ F$, that
is, $F$ is a linear automorphism such that $F^*\colon
\frg^*\longrightarrow \frg^*$ commutes with the Chevalley-Eilenberg
differential~$d$ and $F$ commutes with the complex structures $J$
and $J'$. The latter condition is equivalent to say that~$F^*$,
extended to the complexified exterior algebra, preserves the
bigraduations induced by $J$ and $J'$.

Notice that if $\frg^{1,0}_J$ and $\frg^{1,0}_{J'}$ denote the
$(1,0)$-subspaces of $\gc^*$ associated to~$J$ and $J'$, then the complex
structures~$J$ and $J'$ are equivalent if and only if
there is a $\mathbb{C}$-linear isomorphism $F^*\colon \frg^{1,0}_J
\longrightarrow \frg^{1,0}_{J'}$ such that $d \circ F^*=F^* \circ
d$.

\smallskip

In dimension 6, by Theorem~\ref{clasif-complex}, if the NLA $\frg$ admits complex structures
then all of them are either nilpotent or non-nilpotent.
The classification of abelian complex structures up to equivalence
is obtained in~\cite{ABD}, whereas the non-nilpotent complex structures
are classified in \cite{UV1} (see Section~\ref{ClassifGeneral} for details).
Therefore, it remains to study the equivalence classes of non-abelian nilpotent complex structures.
In order to provide such classification, our starting point is the following reduction of the nilpotent condition \eqref{nilpotent-condition}.

\begin{proposition}\label{J-red}\cite{U}
Let $J$ be a nilpotent complex structure on an NLA $\frg$ of dimension $6$.
There is a basis
$\{\omega^j\}_{j=1}^3$ for $\frg^{1,0}$ satisfying
\begin{equation}\label{nilpotentJ}
\left\{
\begin{array}{lcl}
d\omega^1 \zzz & = &\zzz 0,\\
d\omega^2 \zzz & = &\zzz \epsilon\, \omega^{1\bar{1}} \, ,\\
d\omega^3 \zzz & = &\zzz \rho\, \omega^{12} + (1-\epsilon)A\,
\omega^{1\bar{1}} + B\, \omega^{1\bar{2}} + C\, \omega^{2\bar{1}} +
(1-\epsilon)D\, \omega^{2\bar{2}},
\end{array}
\right.
\end{equation}
where $A,B,C,D\in \mathbb{C}$ and $\epsilon,\rho \in \{0,1\}$.
\end{proposition}

Here $\omega^{jk}$
(resp. $\omega^{j\overline{k}}$) means the wedge product
$\omega^j\wedge\omega^k$ (resp.
$\omega^j\wedge\omega^{\overline{k}}$), where
$\omega^{\overline{k}}$ indicates the complex conjugated of
$\omega^k$. From now on, we shall use a similar abbreviated notation
for ``basic'' forms of arbitrary bidegree.

Notice that in the equations \eqref{nilpotentJ} the complex structure is not abelian if and only if $\rho=1$.
Next we study the 2-step and 3-step cases in Sections~\ref{2step} and~\ref{3step}, respectively.

\subsection{Non-abelian complex structures in the 2-step case}\label{2step}

\noindent
Any 6-dimensional 2-step NLA $\frg$ has first Betti number at least 3, and if it is equal to 3 then necessarily the coefficient $\epsilon$ in \eqref{nilpotentJ} is non-zero. We consider firstly $\epsilon=0$, i.e. the Lie algebra has first Betti number $\geq 4$, and
we will finish the section by considering the remaining case $\epsilon=1$.

The following proposition provides a further reduction of the equations \eqref{nilpotentJ}
when $\epsilon=0$ and the structure is not complex-parallelizable. Recall that $J$ is \emph{complex-parallelizable} if
$[JX,Y]=J[X,Y]$, for all $X,Y\in \frg$, or equivalently $d(\frg^{1,0})\subset\bigwedge^{2,0}(\frg^*)$. These structures
are the natural complex structures of \emph{complex Lie algebras}, and in six dimensions they correspond to
$\epsilon=A=B=C=D=0$ and the possible Lie algebras are
$\frh_1$ (for $\rho=0$) and $\frh_5$ (for $\rho=1$).

\begin{proposition}\label{nueva}
Let $J$ be a complex structure on a $2$-step NLA $\frg$ of
dimension~$6$ with first Betti number~$\geq 4$. If $J$ is not
complex-parallelizable, then there is a basis $\{\omega^j\}_{j=1}^3$
for $\frg^{1,0}$ such that
\begin{equation}\label{epsilonzero-red-new-bis}
d \omega^1=d\omega^2=0,\quad d\omega^3=\rho\, \omega^{12} +
\omega^{1\bar{1}} + \lambda\,\omega^{1\bar{2}} + D\,\omega^{2\bar{2}},
\end{equation}
where $\rho\in \{0,1\}$, $\lambda\in \mathbb{R}$ such that
$\lambda\geq 0$, and $D\in \mathbb{C}$ with $\Imag D\geq 0$.
Moreover, if we denote $x=\Real D$ and
$y=\Imag D$, then:
\begin{enumerate}
\item[{\rm (i)}] If $\lambda=\rho$, then the Lie algebra $\frg$ is
isomorphic to
\begin{enumerate}
\item[{\rm (i.1)}] $\frh_2$, for $y>0$;
\item[{\rm (i.2)}]
$\frh_3$, for $\rho=y=0$ and $x\not=0$;
\item[{\rm (i.3)}] $\frh_4$,
for $\rho=1$, $y=0$ and $x\not=0$;
\item[{\rm (i.4)}] $\frh_6$, for
$\rho=1$ and $x=y=0$;
\item[{\rm (i.5)}] $\frh_8$, for $\rho=x=y=0$.
\end{enumerate}
\item[{\rm (ii)}] If $\lambda\not=\rho$, then the Lie algebra $\frg$
is isomorphic to
\begin{enumerate}
\item[{\rm (ii.1)}] $\frh_2$, for $4y^2 >
(\rho-\lambda^2)(4x+\rho-\lambda^2)$;
\item[{\rm (ii.2)}] $\frh_4$, for
$4y^2 = (\rho-\lambda^2)(4x+\rho-\lambda^2)$;
\item[{\rm (ii.3)}] $\frh_5$,
for $4y^2 < (\rho-\lambda^2)(4x+\rho-\lambda^2)$.
\end{enumerate}
\end{enumerate}
\end{proposition}

\begin{proof}
In \cite[Lemma 11]{U} it is proved that under these conditions there is a basis $\{\sigma^j\}_{j=1}^3$
for $\frg^{1,0}$ such that
\begin{equation}\label{epsilonzero-red-new}
d \sigma^1=d\sigma^2=0,\ d\sigma^3=\rho\, \sigma^{12} +
\sigma^{1\bar{1}} + B\,\sigma^{1\bar{2}} + D\,\sigma^{2\bar{2}},
\end{equation}
where $B,D\in \mathbb{C}$ and $\rho\in \{0,1\}$.

If $B\neq0$ then we can take any non-zero
solution $z$ of $\bar z\frac{B}{|B|}=z$, and the equations~\eqref{epsilonzero-red-new} reduce to
\eqref{epsilonzero-red-new-bis} with $\lambda=|B|$
with respect to the new basis
$\{\omega^1=z\,\sigma^1, \omega^2=\bar z\,\sigma^2, \omega^3=|z|^2\,\sigma^3\}$.

Consider now $B=\lambda$ with
$\lambda\in \mathbb{R}^{\geq 0}$ in \eqref{epsilonzero-red-new}. If $D\neq 0$, then
with respect to the new basis
$\{\omega^1=-\bar D\,\sigma^2, \omega^2=\sigma^1+\lambda\,\sigma^2,\omega^3=\bar D\,\sigma^3\}$
we get \eqref{epsilonzero-red-new-bis} with $\bar D$ instead of~$D$.

Finally, the second part of the proposition follows directly from \cite[Proposition~13]{U}.
\end{proof}

From now on we consider $\rho=1$.
By Proposition~\ref{nueva} any two complex
structures on the Lie algebra $\frh_6$ are equivalent.
Thus, it remains to classify up to equivalence the non-abelian structures $J$
on $\frh_2$, $\frh_4$ and $\frh_5$. Any such $J$ is identified
with a triple $(1,\lambda, D)$ through equations \eqref{epsilonzero-red-new-bis} with $\rho=1$,
$\lambda\geq 0$ and $\Imag D\geq 0$.

We will say that two triples $(1,\lambda, D)$ and $(1,\lambda', D')$ are equivalent, denoted by $(1,\lambda, D)\sim (1,\lambda', D')$,
if the corresponding structures $J$ and $J'$ are equivalent. So, the problem reduces to classify triples $(1,\lambda, D)$
up to equivalence.

\begin{lemma}\label{lema-eq}
Let us consider two triples $(1,\lambda, D)$ and $(1,t,E)$ as above.
\begin{enumerate}
\item[{\rm (i)}] If $D=0$ then, $(1,t,E)\sim (1,\lambda,0)$ if and only if $t=\lambda$ and $E=0$.
\item[{\rm (ii)}] If $D\not=0$ then, $(1,t,E)\sim (1,\lambda,D)$ if and only if
there exist non-zero complex numbers $e,\,f$ such that $E=De/\bar{e}$ and
\begin{equation}\label{condicion_final}
\left(\frac{|f|^2}{\bar e}-1\right) (\bar D\bar e-De)^2
=(\lambda \bar f-tf)(\lambda\bar D\bar e f-tDe\bar
f).\end{equation}
\end{enumerate}
\end{lemma}

\begin{proof}
The structure equations corresponding to the triples $(1,\lambda, D)$ and $(1,t,E)$ are
$$d\omega^1=d\omega^2=0,\quad d\omega^3=\omega^{12} +
\omega^{1\bar1}+\lambda\omega^{1\bar2}+D\omega^{2\bar2},$$
$$d\sigma^1=d\sigma^2=0,\quad d\sigma^3=\sigma^{12} +
\sigma^{1\bar1}+t\sigma^{1\bar2}+E\sigma^{2\bar2},$$
where $\lambda,t\geq 0$ and $\Imag D, \Imag E \geq 0$.
Then
$(1,t,E)\sim (1,\lambda,D)$ if and only if there exists an automorphism of the Lie algebra preserving the complex equations, i.e.
there is $(m_{ij})\in {\rm GL}(3,\mathbb{C})$ such that
$\sigma^i=\sum_{j=1}^3 m_{ij}\,\omega^j$ and
$$
d\sigma^i=\sum_{j=1}^3 m_{ij}\,d\omega^j,\quad i=1,2,3.
$$
These conditions are equivalent to
$$
\sigma^1=a\,\omega^1+b\,\omega^2,\quad
\sigma^2=c\,\omega^1+f\,\omega^2,\quad
\sigma^3=m_{31}\,\omega^1+m_{32}\,\omega^2+ e\,\omega^3,
$$
and
\begin{equation}\label{ecuaciones_cambio}
\begin{array}{rl}
\begin{cases}
(\textrm{I})&e= af-bc,\\
(\textrm{II})&e=|a|^2 + t\,a\bar c + E|c|^2,\\
(\textrm{III})&\lambda e= a\bar b + t\,a\bar f + E c\bar f,\\
(\textrm{IV})&0= \bar ab + t\,b\bar c + E \bar cf,\\
(\textrm{V})&De=|b|^2 + t\,b\bar f + E |f|^2.
\end{cases}
\end{array}
\end{equation}
Notice that $m_{13}=m_{23}=0$, $e\not=0$ and the coefficients $m_{31}$ and $m_{32}$ are not relevant.

It is straightforward to see that coefficient $f$ must be
non-zero (otherwise $\lambda=t$ and $D=E$)
and so we can express $a$ as
$$a=\frac{e+bc}{f}.$$

First of all, let us suppose that $D=0$.  Replacing $a$ in $(\textrm{IV})$ and using $(\textrm{V})$ we obtain that $b=0$
and therefore $E=0$ by equation $(\textrm{V})$.
Combining $(\textrm{I})$ and $(\textrm{III})$ we get that $\lambda f= t\bar f$.
Since $\lambda$ and $t$ are real non-negative numbers, we conclude that $\lambda=t$,
i.e. $(1,\lambda,0)$ defines an equivalence class for every $\lambda\geq0$.
This completes the proof of~(i).

We suppose next that $D\neq0$. In order to solve~\eqref{ecuaciones_cambio} we transform it into an equivalent system by doing
the following substitutions.
Replacing $a$ in equation $(\textrm{IV})$ and using $(\textrm{V})$ we can express
$$\bar c=-\frac{b \bar e}{D e}.$$
Next, in $(\textrm{II})$ we can substitute $a$ and $c$ and use again $(\textrm{V})$ to obtain that
$$De=E\bar e,$$
which implies in particular $|D|=|E|$. Notice that since $D\not=0$ we can assume $E\neq \bar D$ by Proposition~\ref{nueva}.
Now, $\bar c=-b/E$.
Proceeding in a similar way in equation $(\textrm{III})$ we get
$$\bar b=\frac{\lambda f-t\bar f}{1-D/\bar E}.$$
Finally, using the expressions of $a$, $b$, $c$ above, equation $(\textrm{V})$ is equivalent to \eqref{condicion_final}.
Therefore, given $e,f\in\C-\{0\}$ satisfying $De=E\bar e$ and \eqref{condicion_final}, it is always possible to find $a,b,c\in\C$ such that system~\eqref{ecuaciones_cambio} is satisfied.
\end{proof}

\begin{remark}
{\rm
As a consequence of Lemma~\ref{lema-eq}~(ii), when $D\not=0$ a
necessary condition for
$(1,t, E)$ to be equivalent to $(1,\lambda, D)$ is that $|D|=|E|$.
Moreover, to find an equivalent complex structure $(1,t,E)$ it
suffices to find $t\geq 0$ and $e,f\in\C - \{0\}$ satisfying~\eqref{condicion_final},
because $E$ is necessarily given by $E=De/\bar{e}$.
}
\end{remark}

\begin{corollary}\label{no-equiv-lambda-iguales}
Let $E\neq \bar D$.
If $(1,t, E)\sim (1,\lambda, D)$ then, $t=\lambda$ if and only if~$E=D$.
\end{corollary}

\begin{proof}
By hypothesis $D$ cannot be zero, so we are in case (ii) of Lemma~\ref{lema-eq}.
Suppose first that $\lambda=t$ in \eqref{condicion_final}, i.e.
$$(\bar D\bar e-De)^2\left(\frac{|f|^2}{\bar
e}-1\right)=\lambda^2(\bar f-f)(\bar D\bar e f-De\bar f).$$
The right hand side of the previous equality is a real number.
If it
is zero then $e=|f|^2$ (otherwise $De=\bar D\bar e$ would imply $E=\bar
D$); thus, $e$ is a real number and since $E=De/\bar{e}$ we conclude that $D=E$.
On the other hand, if it is a
non-zero real number, then $\frac{|f|^2}{\bar e}-1$ must be a
real number and then $e\in\mathbb R$ and again~$D=E$.

Conversely, let us suppose that $E=D\neq 0$.
In this case $e\in\R$ and by \eqref{condicion_final} we can
express it as
$$e=|f|^2-\frac{(\lambda \bar f - tf)(\lambda \bar D f-tD\bar f)}{(\bar
D-D)^2}.$$
Notice that by hypothesis $D\not=\bar E=\bar D$. To ensure that $e\in\R$ it must
happen that $(\lambda \bar f - tf)(\lambda
\bar D f-tD\bar f)\in\R$ or equivalently,
$$|f|^2(\lambda^2-t^2)(\bar D-D)=0.$$
As $f(\bar D-D)\neq 0$ the only possibility
to solve the previous equation is $\lambda= t$.
\end{proof}

From the previous results it follows that it remains to consider the case when $D\neq 0$ and $\lambda\neq t$.
The next lemma provides a simplification of equation \eqref{condicion_final}.

\begin{lemma}\label{nec-suf-condicion}
Let us suppose that $\lambda\neq t$, $D=x+iy\neq 0$ and $e\in \mathbb{C}-\{0\}$. Then,
$(1,\lambda,D)\sim (1,t,De/\bar{e})$ if and only if
\begin{equation}\label{condicion-necesaria}
4y^2 - (t^2-\lambda^2)(4x+t^2-\lambda^2)\geq 0.
\end{equation}
\end{lemma}

\begin{proof}
By Lemma~\ref{lema-eq}~(ii), we know that $(1,\lambda,D)\sim (1,t,De/\bar{e})$ if and only if \eqref{condicion_final}
is satisfied. This condition reads, with respect to $H=De$, as
\begin{eqnarray*}
(\bar H-H)^2\left(\bar D|f|^2-\bar H\right)=\bar H(\lambda \bar
f-tf)(\lambda f\bar H-t\bar f H).
\end{eqnarray*}
Taking real and imaginary parts in the expression above we obtain
\begin{equation}\label{real-im-A}
\begin{array}{ll}
\begin{cases}
4H_2^2(H_1-x|f|^2) =\!\! &|f|^2(t^2-\lambda^2)H_2^2 +
|f|^2(t^2+\lambda^2)H_1^2\\[6pt]
& - 2\lambda t(f_1^2-f_2^2)H_1^2 - 4\lambda tH_1H_2f_1f_2,\\[8pt]
4H_2^2(y|f|^2-H_2) =\!\! &2\lambda H_2\left[tH_1(f_1^2-f_2^2) +
2tH_2f_1f_2 - \lambda |f|^2H_1\right],
\end{cases}
\end{array}\end{equation} where $H=H_1 + iH_2$ and $f=f_1 +
if_2$. Observe that $H_2\neq 0$, otherwise we get a
contradiction using the first equation of \eqref{real-im-A}.

Substituting the second equation of \eqref{real-im-A} in the first
one and replacing $H$ by $De$, we can express the system
\eqref{real-im-A} as
\begin{eqnarray}\label{real-im}
\begin{cases}
e_1^2(t^2-\lambda^2) + 4y e_1e_2 + e^2_2(t^2-\lambda^2 + 4x)=0,\\[5pt]
2H_2(y|f|^2-H_2) =\lambda \left[tH_1(f_1^2-f_2^2) + 2tH_2f_1f_2
- \lambda |f|^2H_1\right],
\end{cases}\end{eqnarray} where $e=e_1 + i e_2$.

To solve the first equation in \eqref{real-im} as a second degree equation in $e_1$ we need the
discriminant to be greater than or equal to 0, i.e.
$4y^2 - (t^2-\lambda^2)(4 x+t^2-\lambda^2)\geq 0$, which is precisely condition~\eqref{condicion-necesaria}.

Now, suppose that \eqref{condicion-necesaria} holds.
Then we
obtain that
$$e_1=\frac{e_2\beta}{\lambda^2-t^2},\quad\quad\quad e=e_2\left(\frac{\beta}{\lambda^2-t^2} + i\right),$$
where
$\beta=2y + \sqrt{4y^2 - (t^2-\lambda^2)(4x+t^2-\lambda^2)}$ and
$e_2$ is determined by the second equation in \eqref{real-im}.
\end{proof}

\begin{corollary}\label{equiv-general}
Let us suppose that $\lambda\neq t$ and $D=x+iy\neq 0$.
If \eqref{condicion-necesaria} holds then
$$(1, \lambda, D)\sim \left(1, t, D\left(\displaystyle\frac{\beta^2-(\lambda^2-t^2)^2}{\beta^2+(\lambda^2-t^2)^2} +
\displaystyle\frac{2\beta(\lambda^2-t^2)}{\beta^2+(\lambda^2-t^2)^2}\,i\right)\right),$$
where $\beta=2y + \sqrt{4y^2 -
(t^2-\lambda^2)(4x+t^2-\lambda^2)}$.
\end{corollary}

Comparing the inequalities (ii.1) and (ii.2) in Proposition~\ref{nueva} with the condition~\eqref{condicion-necesaria},
we observe that for $\frh_2$ and $\frh_4$ it is possible to take $t=1$ in the previous corollary in order to get equivalences with
the complex structures (i.1) and (i.3), respectively.
Therefore, using Corollary~\ref{no-equiv-lambda-iguales}, we conclude:

\begin{proposition}\label{clasif-h2-y-h4}
Let us consider the family of complex structures
\begin{equation}\label{h2-h4}d\omega^1=d\omega^2=0,\quad d\omega^3=\omega^{12} +
\omega^{1\bar1}+\omega^{1\bar2}+D\,\omega^{2\bar2},\quad \Imag D\geq 0.\end{equation}
Then:
\begin{enumerate}
\item[(i)]  Any non-abelian complex structure on $\frh_2$ is  equivalent to one and only one structure in \eqref{h2-h4} with $\Imag D>0$;
\item[(ii)] Any non-abelian complex structure on $\frh_4$ is  equivalent to one and only one structure in \eqref{h2-h4} with $D\in\R-\{0\}$.
\end{enumerate}
\end{proposition}

The classification of complex structures on $\frh_5$ requires a more subtle study.

\begin{lemma}
Any non-abelian complex structure on $\frh_5$ which is not complex-parallelizable
belongs to one of the following families: \begin{itemize}
\item[(I)] $d\omega^1=d\omega^2=0,\ \ d\omega^3=\omega^{12} +
\omega^{1\bar 1} + \lambda\,\omega^{1\bar 2} +
iy\,\omega^{2\bar2}, \ \ \textrm{where}\,\,\,
0\leq 2y<|1-\lambda^2|;$
\item[(II)] $d\omega^1=d\omega^2=0,\ \
d\omega^3=\omega^{12} + \omega^{1\bar 1} + (x+iy)\,\omega^{2\bar2},
\ \  \textrm{where}\,\,\, 4y^2<1+4x.$
\end{itemize}
Moreover,
\begin{itemize}
\item[(i)] the structures in family {\rm (I)} are non-equivalent;
\item[(ii)] the structures in family {\rm (II)} are non-equivalent;
\item[(iii)] a structure $(1,\lambda, iy)$ in family {\rm (I)} is equivalent to a structure in family~{\rm (II)}
if and only if $2\lambda^2\in\left[0,1\right)$ and $2y\in[\lambda^2,1-\lambda^2)$.
\end{itemize}
\end{lemma}

\begin{proof}
Let us consider a complex structure given by $(1,\lambda, D=x+i\,y)$ on $\frh_5$, i.e.
$$
4y^2<(1-\lambda^2)(4x + 1 - \lambda^2),
$$
according to Proposition~\ref{nueva}~(ii.3).
If $\lambda^2\geq 2x$, then $(1, \lambda, D) \sim (1, \sqrt{\lambda^2-2x}, i|D|)$ because \eqref{condicion-necesaria} expresses simply as $4|D|^2\geq 0$ and it trivially holds. On the other hand, if $\lambda^2< 2x$, then $(1, \lambda, D) \sim (1, 0, E)$, where $E$ is given in Corollary~\ref{equiv-general}, because in this case $4y^2+\lambda^2(4x-\lambda^2)\geq 0$, that is, condition \eqref{condicion-necesaria} is satisfied.

To study further equivalences, it is clear that structures in family (I) are non-equivalent and the same holds for structures in family (II).
Now let us consider the triples $(1,\lambda, iy)$ and $(1,0,E)$. Then, \eqref{condicion-necesaria} expresses simply as
\begin{equation}\label{condicion-necesaria-reducida-h5}
4y^2\geq \lambda^4.
\end{equation}
Condition for family (I) implies that $4y^2<(1-\lambda^2)^2$, which is equivalent to $4y^2-\lambda^4<1-2\lambda^2$, so if $2\lambda^2\geq 1$ then \eqref{condicion-necesaria-reducida-h5} does not hold.  Now, if $0\leq \lambda^2<\frac 12$ then the condition for family (I) is
equivalent to $y<\frac 12-\frac{\lambda^2}{2}$, and therefore when $2y\in[\lambda^2, 1-\lambda^2)$ the triple $(1,\lambda, iy)$ in family (I)
is equivalent to the triple $(1,0,E=-\frac 12(\lambda^2 - \sqrt{4y^2-\lambda^4}\,i))$ in family (II).
\end{proof}

\begin{proposition}\label{teo-h5}
Any non-abelian complex structure on $\frh_5$ which is not complex-parallelizable
is equivalent to one and only one structure in the following families:
\begin{itemize}
\item[(I)] $d\omega^1=d\omega^2=0,\quad d\omega^3=\omega^{12} +
\omega^{1\bar 1} + \lambda\,\omega^{1\bar 2} +
D\,\omega^{2\bar2},$ \\[7pt]
where $\Real D=0$ and
$\begin{cases}
0\leq 2\,\Imag D<\lambda^2,\quad\ 0<\lambda^2<\frac 12; \mbox{ or }\\[6pt]
0\leq 2\,\Imag D<|1-\lambda^2|,\quad
\frac 12\leq \lambda^2.
\end{cases}$
\item[(II)] $d\omega^1=d\omega^2=0,\quad d\omega^3=\omega^{12} +
\omega^{1\bar 1} + D\,\omega^{2\bar2}, \quad \textrm{where}\,\,\,
4(\Imag D)^2<1+4\,\Real D$.
\end{itemize}
\end{proposition}

\medskip

To finish this section, it remains to study the case of 2-step NLAs $\frg$ with first Betti number equal to 3,
which corresponds to $\epsilon=1$ in \eqref{nilpotentJ}.

\begin{proposition}\label{h7}
Let $J$ be a nilpotent complex structure on an NLA $\frg$ given by~\eqref{nilpotentJ} with $\epsilon=1$, i.e.
$$
d\omega^1 = 0,\quad d\omega^2 = \omega^{1\bar{1}}, \quad d\omega^3 =
\rho\, \omega^{12} + B\, \omega^{1\bar{2}} + C\, \omega^{2\bar{1}},
$$
with $\rho\in \{0,1\}$ and $B,C\in \mathbb{C}$ such that $(\rho,B,C)\not=(0,0,0)$.
Then $\frg$ is $2$-step nilpotent if and
only if $B=\rho=1$ and $C=0$. In such case $\frg$ is isomorphic to
$\frh_7$ and all the complex structures are equivalent.
\end{proposition}

\begin{proof}
Let $Z_1,Z_2,Z_3$ be the dual basis of $\omega^1,\omega^2,\omega^3$.
It is clear that $[\frg,\frg]$ has dimension at least 2 and is contained
in $\langle i(Z_2-\bar{Z}_2), \Real Z_3, \Imag Z_3 \rangle$.
Since $\Real Z_3, \Imag Z_3$ are central elements
and
$$
[i(Z_2-\bar{Z}_2), Z_1]= (\rho-B)i\, Z_3 + \bar{C}i\, \bar{Z}_3,
$$
we conclude that $\frg$ is 2-step nilpotent if and only if $B=\rho$ and
$C$ vanishes.

Let $(\rho, B, C)=(1, 1, 0)$ and let us consider a basis $\{e^1,\ldots, e^6\}$ for $\frg^*$
given by $\omega^1=\frac{1}{\sqrt2}(e^2 + i e^1)$, $\omega^2 = \frac{1}{\sqrt2}e^3 + i e^4$
and $\omega^3 = e^6 + i e^5$.  Now, the Lie algebra $\frg$ is isomorphic to $\frh_7$.
\end{proof}

\subsection{Nilpotent complex structures in the 3-step case} \label{3step}

\noindent
In this section we classify, up to equivalence, nilpotent complex structures on 3-step NLAs $\frg$ of dimension~6.
In this case the coefficient $\epsilon=1$ in the equations~\eqref{nilpotentJ} given in Proposition~\ref{J-red}.
The equivalence of complex structures in terms of the triple $(\rho,B,C)$ is given in the following lemma.

\begin{lemma}\label{nilp-up-to-equiv}
Let $\frg$ be an NLA endowed with a nilpotent complex
structure~$\eqref{nilpotentJ}$ with $\epsilon=1$ and
$(\rho,B,C)\not=(0,0,0)$. Then:
\begin{enumerate}
\item[{\rm (i)}] If the structure is abelian, then
there is a basis $\{\omega^j\}_{j=1}^3$ for $\frg^{1,0}$ satisfying either
\begin{equation}\label{reduced-abel-1}
d\omega^1 = 0,\quad d\omega^2 = \omega^{1\bar{1}}, \quad d\omega^3
=\omega^{2\bar{1}},
\end{equation}
or
\begin{equation}\label{reduced-abel-2}
d\omega^1 = 0,\quad d\omega^2 = \omega^{1\bar{1}}, \quad d\omega^3 =
\omega^{1\bar{2}} + c\, \omega^{2\bar{1}},
\end{equation}
where $c\in \mathbb{R}$, $c\geq 0$.
\item[{\rm (ii)}] In the non-abelian case there is a basis $\{\omega^j\}_{j=1}^3$
for $\frg^{1,0}$ satisfying
\begin{equation}\label{reduced-noabel-nilp}
d\omega^1 = 0,\quad d\omega^2 = \omega^{1\bar{1}}, \quad d\omega^3 =
\omega^{12} + B\, \omega^{1\bar{2}} + c\, \omega^{2\bar{1}},
\end{equation}
where $B\in \mathbb{C}$ and $c\in \mathbb{R}$ such that $c\geq 0$.
\end{enumerate}
Moreover, for any possible choice of parameters $B$ and $c$, each
structure in~\eqref{reduced-abel-1}, \eqref{reduced-abel-2}
and \eqref{reduced-noabel-nilp}
defines an equivalence class of complex structures.
\end{lemma}

\begin{proof}
If the complex structure is abelian then the pair $(B,C)\not= (0,0)$
since $\rho=0$. If $B=0$ then it is clear that one arrives at
equation~(\ref{reduced-abel-1}). If $B\not=0$ then
with respect to the basis $\{ z\, \omega^1, |z|^2\, \omega^2,
{z|z|^2\over B} \, \omega^3 \}$, where $z$ is any non-zero solution
of ${|C| \over |B|}\, \bar{z}= {C\over B}\, z$, the
equations~(\ref{nilpotentJ}) reduce to the form~(\ref{reduced-abel-2}).

For the proof of (ii), we observe that with respect
to $\{ z\, \omega^1, |z|^2\, \omega^2, z|z|^2\, \omega^3 \}$,
where $z\not=0$ satisfies $\bar{z}\, |C|= z\, C$, the
equations~\eqref{nilpotentJ} reduce to~\eqref{reduced-noabel-nilp}.

Finally, the non-equivalence of the different complex structures defined
in~\eqref{reduced-abel-1}, \eqref{reduced-abel-2} and \eqref{reduced-noabel-nilp}
follows by a similar argument to the first part of the proof of Lemma~\ref{lema-eq}.
\end{proof}

The following result provides a classification of abelian structures in
the 3-step case in a slightly more straightforward way than the one given in \cite{ABD}.

\begin{corollary}\label{abelian:h9,h15}
Let $J$ be an abelian structure on an NLA $\frg$ given
by~\eqref{reduced-abel-1} or~\eqref{reduced-abel-2}. Then, $\frg$ is
isomorphic to $\frh_{15}$, except for $c=1$
in which case $\frg\cong \frh_9$.
\end{corollary}

\begin{proof}
For the equations~\eqref{reduced-abel-2}, let us consider a basis
$\{e^1,\ldots,e^6 \}$ for $\frg^*$ given by $\omega^1 = -e^1 +
i\, e^2$, $\omega^2 = 2e^3 + 2i\, e^4$ and $\omega^3=2e^5 + 2(c+1)i\, e^6$.
Then, $e^1,e^2,e^3$ are closed,
$de^4=e^{12}$, $de^5=(c-1)(e^{13}+e^{42})$
and $de^6=e^{14}+e^{23}$. Thus, if $c\not= 1$ then
the Lie algebra $\frg$ is isomorphic to $\frh_{15}$; otherwise,
$\frg\cong\frh_9$. Finally, it is easy to check that the Lie algebra
$\frg$ underlying~(\ref{reduced-abel-1}) is also isomorphic to
$\frh_{15}$.
\end{proof}

Notice that the family~(\ref{reduced-noabel-nilp}) includes the case $\frh_7$
precisely for $\rho=B=1$ and $c=0$ as it is shown in Proposition~\ref{h7}.
Next we determine the Lie algebras underlying the complex equations~(\ref{reduced-noabel-nilp})
in the remaining cases. They all have first Betti number
equal to 3 and are nilpotent in step $3$. Also notice that
the dimension of their center is at least 2.

\begin{proposition}\label{h16}
Let $J$ be a nilpotent complex structure on a $3$-step NLA $\frg$
given by~\eqref{reduced-noabel-nilp}. Then $\frg$ has $3$-dimensional
center if and only if $|B|=1$, $B\not=1$ and $c=0$. In such case
$\frg$ is isomorphic to $\frh_{16}$.
\end{proposition}

\begin{proof}
Let $Z_1,Z_2,Z_3$ be the dual basis of $\omega^1,\omega^2,\omega^3$.
Then, $\Real(Z_3)$ and $\Imag(Z_3)$ are central elements. Let $T=
\lambda_1 Z_1 + \bar{\lambda}_1 \bar{Z}_1 + \lambda_2 Z_2 +
\bar{\lambda}_2 \bar{Z}_2$ be another non-zero element in the center
of $\frg$, where $(\lambda_1,\lambda_2) \in \mathbb{C}^2-\{(0,0)\}$. It
follows from~(\ref{reduced-noabel-nilp}) that
$$
0=[T,Z_1]= \bar{\lambda}_1 Z_2 -  \bar{\lambda}_1 \bar{Z}_2 -
(\lambda_2- B \bar{\lambda}_2) Z_3 - c \bar{\lambda}_2 \bar{Z}_3,
$$
which implies $\lambda_1=0$, $c\lambda_2=0$ and $\lambda_2=B
\bar{\lambda}_2$. Therefore, $c=0$ and $|B|=1$ in order the center
to be 3-dimensional, because otherwise the equation
$\lambda_2=B\bar{\lambda}_2$ would have trivial solution. Moreover, $B\not=1$
because $\frg$ is nilpotent in step 3.

Finally, since $|B|=1$ and $B\not=1$, let us consider the basis $\{
e^1,\ldots,e^6 \}$ for $\frg^*$ given by: $e^1+i\,
e^2 = i(B-1) \omega^1$, $e^3 = \omega^2 +
\omega^{\bar{2}}$, $e^4 = {1-\Real B\over 1-B} i (\omega^2 + B\,
\omega^{\bar{2}})$, $e^5 + i\, e^6 = (1-\Real B) \omega^3$.
Then, we can write the differential of $\omega^3$ in the form
$$
d\omega^3 = \omega^1\wedge(\omega^2+B\,\omega^{\bar{2}}) = \left(
{i(B-1)\over 1-\Real B} \omega^1 \right) \wedge \left( {1-\Real B\over
1-B} i (\omega^2 + B\, \omega^{\bar{2}}) \right),
$$
which implies that $e^1,e^2,e^3$ are closed,
$de^4=e^{12}$, $de^5=e^{14}$ and
$de^6=e^{24}$, i.e. $\frg\cong \frh_{16}$.
\end{proof}

Next we establish the conditions for the coefficients $B$ and $c$
in terms of the dimension of $\frg^2=[\frg,[\frg,\frg]]$.

\begin{lemma}\label{dim-g2}
Let $J$ be a complex structure on a $3$-step NLA $\frg$ given
by~\eqref{reduced-noabel-nilp}. Then:
\begin{enumerate}
\item[{\rm (i)}] If $c=|B-1|\not=0$, then $\dim \frg^2=1$.
\item[{\rm (ii)}] If $c\not=|B-1|$, then $\dim \frg^2=2$.
\end{enumerate}
\end{lemma}

\begin{proof}
From~(\ref{reduced-noabel-nilp}) we have that
$$
\frg^2=[Z_2-\bar{Z}_2, \frg] = \langle (1-B)Z_3 + c\,\bar{Z}_3,\
c\,Z_3 + (1-\bar{B}) \bar{Z}_3 \rangle.
$$
It is clear that $\dim \frg^2=2$ if and only if
$(1-B)(1-\bar{B})-c^2\not=0$.
\end{proof}

Notice that if $c=|B-1|\not=0$ then $\frg$ is isomorphic to $\frh_{10}$,
$\frh_{11}$ or $\frh_{12}$.
Since the case $c=0\not=|B-1|$, $|B|=1$ corresponds to
$\frg\cong\frh_{16}$ by Proposition~\ref{h16}, we conclude that for $c\not=|B-1|$ and $(c,|B|)\not=(0,1)$
the Lie algebra $\frg$ is isomorphic to $\frh_{13}$,
$\frh_{14}$ or $\frh_{15}$.

In order to distinguish the underlying Lie algebras, we use the following argument for
$\frg=\frh_k$, $10\leq k\leq 15$.
Let $\alpha(\frg)$ be the number of linearly independent elements $\tau$ in $\bigwedge^2(\frg^*)$ such that $\tau\in d(\frg^*)$ and $\tau\wedge\tau=0$.  This number can be identified with the number of linearly independent exact $2$-forms which are decomposable, that is,
$\alpha(\frh_k)=3$ for $k=10, 12, 13$,  $\alpha(\frh_k)=2$ for $k=11, 14$ and $\alpha(\frh_k)=1$ for $k=15$.

If $\tau$ is any exact element in $\bigwedge^2(\frg^*)$ then $\tau=\mu\,d\omega^2 + \bar\mu\,d\omega^{\bar2} + \nu\,d\omega^3 + \bar\nu\,d\omega^{\bar3}$, for some $\mu,\nu\in\C$, and by~\eqref{reduced-noabel-nilp} we have
$$
\tau = (\mu-\bar\mu)\,\omega^{1\bar1} + \nu\,\omega^{12} + (\nu B-\bar\nu c)\,\omega^{1\bar2} + (\nu c - \bar\nu\bar B)\,\omega^{2\bar1} + \bar\nu\,\omega^{\bar1\bar2}.
$$
A direct calculation shows that
$$\tau\wedge\tau = 2\left(|\nu|^2(1-|B|^2-c^2) + c\,\left(\nu^2B + \bar\nu^2 \bar B\right)\right)\,\omega^{12\bar1\bar2}.$$
Thus, if we denote $p=\Real\,\nu$ and $q=\Imag\,\nu$, then $\tau\wedge\tau=0$ if and only if
\begin{equation}\label{ec-grado2}
\left(1\!-\!|B|^2\!-\!c^2\!+2c\,\Real\,B\right)p^2- \left(4c\,\Imag\,B\right)pq + \left(1\!-\!|B|^2\!-\!c^2\!-2c\,\Real\,B\right)q^2=0.
\end{equation}
Observe that the trivial solution $p=q=0$ corresponds to $\tau = 2i\,\Imag\,\mu\, \omega^{1\bar1}$,
according to the fact that $\alpha(\frg)\geq 1$.

\begin{proposition}\label{h10-11-12}
Let $J$ be a complex structure on a $3$-step NLA $\frg$ given
by~\eqref{reduced-noabel-nilp} with $c=|B-1|\not=0$.
Then:
\begin{enumerate}
\item[{\rm (i)}] $\frg\cong \frh_{10}$ if and only if $B=0$;
\item[{\rm (ii)}] $\frg\cong \frh_{11}$ if and only if $B\in \mathbb{R}-\{0,1\}$;
\item[{\rm (iii)}] $\frg\cong \frh_{12}$ if and only if $\Imag B\not=0$.
\end{enumerate}
In particular, all the complex structures on $\frh_{10}$ are equivalent.
\end{proposition}

\begin{proof}
Since $c=|B-1|\not=0$, it follows from Lemma~\ref{dim-g2}
that $\frg$ is isomorphic to $\frh_{10}$, $\frh_{11}$ or $\frh_{12}$.

Firstly, $\frg\cong\frh_{10}$ if and only if the coefficients in equation~\eqref{ec-grado2} vanish.
In fact, for $\frh_{10}$ we have by Theorem~\ref{clasif-complex} that
$\nu\,d\omega^3 + \bar\nu\,d\omega^{\bar3}\in \langle e^{12},e^{13},e^{14} \rangle$ for any $\nu\in\C$
so any pair $(p,q)\in \R^2$ solves the equation~\eqref{ec-grado2}, which implies the vanishing of
its coefficients. Conversely, if the coefficients $1-|B|^2-c^2+2c\,\Real\,B$, $c\,\Imag\,B$ and $1-|B|^2-c^2-2c\,\Real\,B$
are all zero
then neccesarily $B=0$ and $c=1$, that is, $d\omega^1 =0$, $d\omega^2 =
\omega^{1\bar{1}}$ and $d\omega^3 =(\omega^{1} - \omega^{\bar{1}})\wedge \omega^{2}$, and therefore the Lie algebra is
isomorphic to $\frh_{10}$.

On the other hand, notice that if $c=|B-1|\not=0$ and $(B,c)\not=(0,1)$ then~\eqref{ec-grado2} is a second degree equation in $p$ or $q$.
Since its discriminant is a positive multiple of $(\Imag\,B)^2$,
if $\Imag\,B\not=0$ then we get two independent solutions and $\alpha(\frg)=3$, that is, $\frg\cong \frh_{12}$.
Finally, for $\Imag\,B=0$ the equation~\eqref{ec-grado2} provides one solution and $\alpha(\frg)=2$,
so $\frg\cong \frh_{11}$.
\end{proof}

\begin{proposition}\label{h13-14-15}
Let $J$ be a complex structure on a $3$-step NLA $\frg$ given
by~\eqref{reduced-noabel-nilp} with $c\not=|B-1|$ such that $(c,|B|)\not=(0,1)$.
Then:
\begin{enumerate}
\item[{\rm (i)}] $\frg\cong \frh_{13}$ if and only if $c^4-2(|B|^2+1)c^2+(|B|^2-1)^2<0$;
\item[{\rm (ii)}] $\frg\cong \frh_{14}$ if and only if $c^4-2(|B|^2+1)c^2+(|B|^2-1)^2=0$;
\item[{\rm (iii)}] $\frg\cong \frh_{15}$ if and only if $c^4-2(|B|^2+1)c^2+(|B|^2-1)^2>0$.
\end{enumerate}
\end{proposition}

\begin{proof}
Since $c\not=|B-1|$ and $(c,|B|)\not=(0,1)$, it follows from Lemma~\ref{dim-g2} and Proposition~\ref{h16} that $\frg$ is
isomorphic to $\frh_{13}$, $\frh_{14}$ or $\frh_{15}$.

Notice that the condition $(c,|B|)\not=(0,1)$ implies that the coefficients of $p^2$ and $q^2$
in equation~\eqref{ec-grado2} cannot be both zero, so~\eqref{ec-grado2} is always a second degree equation.
Let
$$\Delta=c^4-2(|B|^2+1)c^2+(|B|^2-1)^2.$$
Since the discriminant as a second degree equation in $p$ is equal to $-4q^2\Delta$
and the discriminant as a second degree equation in $q$ equals $-4p^2\Delta$, the number of
independent solutions of equation~\eqref{ec-grado2} depends on the sign of $\Delta$.
Thus, for $\Delta<0$ there exist two such solutions and thus $\frg\cong \frh_{13}$,
for $\Delta=0$ there exists only one such solution and $\frg\cong \frh_{14}$, and finally
for $\Delta>0$ there is no solution and $\alpha(\frg)=1$, which implies that $\frg\cong \frh_{15}$.
\end{proof}

\section{Classification of complex structures} \label{ClassifGeneral}

\noindent
As a consequence of our previous study, in this section we present in Table 1
the classification of nilpotent complex structures
up to equivalence on 6-dimensional NLAs.
In the table the closed (1,0)-form $\omega^1$ does not appear, and the coefficients
$c,\lambda\in \mathbb{R}^{\geq 0}$ and $B,D\in \mathbb{C}$ with $\Imag D\geq 0$.

In Table 1 we have also included the classification of abelian structures~$J$ on 6-dimensional NLAs obtained in~\cite{ABD}.
In the 3-step case we use directly the equations given in Lemma~\ref{nilp-up-to-equiv} and Corollary~\ref{abelian:h9,h15},
but in the 2-step case we have written the complex structure equations of any abelian $J$ in a form that fits
in our Proposition~\ref{nueva}. More precisely, in the 2-step case we consider first the following reduction of the equations \eqref{epsilonzero-red-new-bis} of any abelian complex structure.

\begin{corollary}\label{e=0}
If $J$ is abelian and $\frg$ is $2$-step then
there is a basis $\{\omega^j\}_{j=1}^3$ for~$\frg^{1,0}$
satisfying one of the following equations:
\begin{enumerate}
\item[{\rm (i)}] $d\omega^1 = d\omega^2 = d\omega^3 =0$;
\item[{\rm (ii)}] $d\omega^1 = d\omega^2 =0$, \ $d\omega^3=
\omega^{1\bar{1}} + D\,\omega^{2\bar{2}}$, with  $D\in \mathbb{C},\ |D|=1, \ \Imag D\geq 0$;
\item[{\rm (iii)}] $d\omega^1 = d\omega^2 =0$, \ $d\omega^3 =
\omega^{1\bar{1}} + \omega^{1\bar{2}} +
D\,\omega^{2\bar{2}}$, with $D\in \mathbb{C}, \ \Imag D\geq 0$.
\end{enumerate}
\end{corollary}

\begin{proof}
Suppose $\rho=0$ in \eqref{epsilonzero-red-new-bis}. If in addition $\lambda=0$, then in terms of the basis
$\{\sqrt{|D|}\;\omega^1,\;|D|\,\omega^2,\;|D|\,\omega^3\}$ we obtain (i) or (ii), whereas if $\lambda\neq0$ then we get
equations (iii) with respect to
$\{\omega^1,\;\lambda\,\omega^2,\;\omega^3\}$.
\end{proof}

Next we illustrate how to rewrite the complex structure equations of any abelian~$J$ on the Lie algebra $\frh_5$
in a form that fits in our Corollary~\ref{e=0}.
By~\cite[Theorem 3.5]{ABD}
there is, up to isomorphism, one family
$J_t$, $t\in (0,1]$, of abelian structures given by
$$
J_t e^1=e^3,\quad J_t e^2=e^4,\quad J_t e^5= \frac{1}{t}\, e^6,
$$
where $\{e^1,\ldots, e^6\}$ is the real basis of $\frh_5$ in Theorem~\ref{clasif-complex}.
Let us consider the basis of $(1,0)$-forms $\{\sigma^1, \sigma^2, \sigma^3\}$ given by
$
\sigma^1 = e^1 - i J_t e^1 = e^1-i\,e^3$,  $\sigma^2 = e^2 - i J_t e^2=e^2-i\, e^4$ and $ \sigma^3 = 2i(e^5 - i J_t e^5) = 2i\, e^5+\frac{2}{t}\, e^6$.  Hence, the complex structure equations for $J_t$ are
$$
d \sigma^1=d \sigma^2=0,\quad
d \sigma^3= \sigma^{1\bar{1}}
-\frac{i}{t} \sigma^{1\bar{2}} -\frac{i}{t} \sigma^{2\bar{1}} -
 \sigma^{2\bar{2}}.
$$
Now, by~\cite[Lemma 11]{U} there exists a (1,0)-basis $\{\omega^j\}_{j=1}^3$ satisfying
$$
d \omega^1=d \omega^2=0,\quad\quad d \omega^3=
\omega^{1\bar{1}}+\omega^{1\bar{2}}+D\,
\omega^{2\bar{2}},$$
with $D=\frac{1-t^2}{4}$. Notice that $D\in [0,\frac{1}{4})$ because $t\in (0,1]$.
Therefore, any abelian complex structure on $\frh_5$ is given, up to isomorphism, as in Table 1.

\medskip

For completeness we include Table 2 with the classification of non-nilpotent complex structures on 6-dimensional NLAs
obtained in \cite{UV1}.  For any of such structures there exists a $(1,0)$-basis $\{\omega^1, \omega^2, \omega^3\}$ satisfying the following structure equations:
\begin{equation}\label{eq-no-nilpot}
d\omega^1 = 0,\quad d\omega^2 =\omega^{13}+\omega^{1\bar3} ,\quad d\omega^3 = i\epsilon\, \omega^{1\bar1}\pm (\omega^{1\bar2}-\omega^{2\bar1}),
\end{equation} where $\epsilon\in\{0,1\}$.

\vfill\eject
{\large\bf Table 1: Classification of nilpotent complex structures}\vspace{7pt}

\hskip-1cm
\begin{tabular}{|c|l|l|}
\hline
$\frg$ & \textsf{Abelian structures ($\rho=0$) } & \textsf{Non-abelian Nilpotent structures ($\rho=1$) }\\
\hline && \\[-7pt]
$\frh_{1}$ &  $d\omega^2=0,\ d\omega^3=0$ &  \hskip3cm{\bf ---} \\
\hline &&\\[-9pt]
$\frh_{2}$ & \hskip-.2cm$
\begin{array}{rl}
\zzz&\zzz d\omega^2 =0,\ d\omega^3=\omega^{1\bar{1}}+
D\, \omega^{2\bar{2}},\\[2pt]
\zzz&\zzz \Imag D=1
\end{array}
$
 & \hskip-.2cm$
\begin{array}{rl}
\zzz&\zzz d\omega^2 =0,\ d\omega^3 =\omega^{12} +
\omega^{1\bar1}+\omega^{1\bar2}+D\,\omega^{2\bar2},\\[2pt]
\zzz&\zzz \Imag D>0
\end{array}
$ \\
\hline &&\\[-9pt]
$\frh_{3}$ & $d\omega^2 = 0,\ d\omega^3
=\omega^{1\bar{1}}\pm \omega^{2\bar{2}}$ & \hskip3cm{\bf ---} \\
\hline &&\\[-9pt]
$\frh_{4}$ & \hskip-.2cm$
\begin{array}{rl}
\zzz&\zzz d\omega^2 =0,\\[2pt]
\zzz&\zzz d \omega^3=
\omega^{1\bar{1}}+\omega^{1\bar{2}}+\frac{1}{4}\omega^{2\bar{2}}
\end{array}
$
& \hskip-.2cm$
\begin{array}{rl}
\zzz&\zzz d\omega^2 =0,\ d\omega^3=\omega^{12} +
\omega^{1\bar1}+\omega^{1\bar2}+D\,\omega^{2\bar2},\\[2pt]
\zzz&\zzz D\in \mathbb{R}\!-\!\{0\}
\end{array}
$ \\[1pt]
\hline &&\\[-9pt]
 &  & $d\omega^2 =0,\ d\omega^3=\omega^{12}$ \\
\cline{3-3}
& & \\[-8pt]
&  & $d\omega^2 =0,\ d\omega^3=\omega^{12} +
\omega^{1\bar1}+\lambda\,\omega^{1\bar 2} +D\,\omega^{2\bar2},$ \\[2pt]
& $d\omega^2=0,$& $\mbox{with } (\lambda,D) \mbox{ satisfying one of:}$ \\[2pt]
$\frh_{5}$& $d \omega^3=\omega^{1\bar{1}}+\omega^{1\bar{2}}+D\,\omega^{2\bar{2}}$,& $\bullet\ \lambda=0\leq  \Imag D, \
4(\Imag D)^2<1+4 \,\Real D$;  \\[2pt]
& $D\in [0,\frac14)$ & $\bullet\  0<\lambda^2<\frac 12,\ 0\leq \Imag D<\frac{\lambda^2}{2},\ \Real D\!=0$;\\[3pt]
& & $\bullet\  \frac 12\leq \lambda^2< 1,\ 0\leq \Imag D<\frac{1-\lambda^2}{2},\ \Real D\!=0$;\\[3pt]
& & $\bullet\  \lambda^2>1,\ 0\leq \Imag D<\frac{\lambda^2-1}{2},\ \Real D\!=0$. \\ [2pt]
\hline && \\[-7pt]
$\frh_{6}$ & \hskip2cm{\bf ---} & $d\omega^2 =0,\ d\omega^3 =\omega^{12} + \omega^{1\bar{1}} + \omega^{1\bar{2}}$ \\[3pt]
\hline &&\\[-9pt]
$\frh_{7}$ & \hskip2cm{\bf ---} & $d\omega^2=\omega^{1\bar{1}},\ d\omega^3 = \omega^{12} + \omega^{1\bar{2}}$ \\[3pt]
\hline &&\\[-9pt]
$\frh_{8}$ &  $d\omega^2 = 0,\ d\omega^3=\omega^{1\bar{1}}$ & \hskip3cm{\bf ---} \\[3pt]
\hline &&\\[-9pt]
$\frh_{9}$ & $d\omega^2=\omega^{1\bar{1}},\ d\omega^3 = \omega^{1\bar{2}} + \omega^{2\bar{1}}$ & \hskip3cm{\bf ---} \\[3pt]
\hline &&\\[-9pt]
$\frh_{10}$ & \hskip2cm{\bf ---} & $d\omega^2 =\omega^{1\bar{1}},\ d\omega^3 =\omega^{12} + \omega^{2\bar{1}}$ \\[3pt]
\hline &&\\[-9pt]
$\frh_{11}$ & \hskip2cm{\bf ---} & \hskip-.2cm$
\begin{array}{rl}
\zzz&\zzz d\omega^2 =\omega^{1\bar{1}},\ d\omega^3 =\omega^{12} + B\, \omega^{1\bar{2}} + |B\!-\!1|\, \omega^{2\bar{1}},\\[2pt]
\zzz&\zzz B\in \mathbb{R}\!-\!\{0,1\}
\end{array}
$ \\ [2pt]
\hline &&\\[-9pt]
$\frh_{12}$ & \hskip2cm{\bf ---} & \hskip-.2cm$
\begin{array}{rl}
\zzz&\zzz d\omega^2 =\omega^{1\bar{1}},\ d\omega^3 =\omega^{12} + B\, \omega^{1\bar{2}} + |B\!-\!1|\, \omega^{2\bar{1}},\\[2pt]
\zzz&\zzz \Imag B\not=0
\end{array}
$ \\
\hline &&\\[-9pt]
$\frh_{13}$ & \hskip2cm{\bf ---} & \hskip-.2cm$
\begin{array}{rl}
\zzz&\zzz d\omega^2 = \omega^{1\bar{1}},\ d\omega^3 = \omega^{12} + B\, \omega^{1\bar{2}} + c\, \omega^{2\bar{1}},\\[3pt]
\zzz&\zzz c\not=|B-1|, \quad (c,|B|)\not=(0,1),\\[3pt]
\zzz&\zzz c^4 - 2(|B|^2+1)c^2+(|B|^2-1)^2 <0
\end{array}
$ \\
\hline &&\\[-9pt]
$\frh_{14}$ & \hskip2cm{\bf ---} & \hskip-.2cm$
\begin{array}{rl}
\zzz&\zzz d\omega^2 = \omega^{1\bar{1}},\ d\omega^3 = \omega^{12} + B\, \omega^{1\bar{2}} + c\, \omega^{2\bar{1}},\\[3pt]
\zzz&\zzz c\not=|B-1|, \quad (c,|B|)\not=(0,1),\\[3pt]
\zzz&\zzz c^4 - 2(|B|^2+1)c^2+(|B|^2-1)^2 =0
\end{array}
$  \\
\hline &&\\[-9pt]
 & $d\omega^2 = \omega^{1\bar{1}},\  d\omega^3 = \omega^{2\bar{1}}$ &
$d\omega^2 = \omega^{1\bar{1}},\ d\omega^3 = \omega^{12} + B\, \omega^{1\bar{2}} + c\, \omega^{2\bar{1}},$ \\
\cline{2-2}
& & \\[-8pt]
$\frh_{15}$& $d\omega^2 = \omega^{1\bar{1}},\ d\omega^3 =\omega^{1\bar{2}} + c\, \omega^{2\bar{1}},$ &
$c\not=|B-1|, \quad (c,|B|)\not=(0,1)$,\\[3pt]
& $c \neq 1$ & $c^4 - 2(|B|^2+1)c^2+(|B|^2-1)^2 >0$ \\
\hline &&\\[-9pt]
$\frh_{16}$ & \hskip2cm{\bf ---} & \hskip-.2cm$
\begin{array}{rl}
\zzz&\zzz d\omega^2 = \omega^{1\bar{1}},\ d\omega^3 = \omega^{12} + B\, \omega^{1\bar{2}},\\[2pt]
\zzz&\zzz |B|=1,\ B\not= 1
\end{array}
$ \\
\hline
\end{tabular}

\vskip.2cm

\noindent $d\omega^1=0$;\ \ $\lambda,c\geq 0$;\ \ $B,D\in \mathbb{C}$.

\newpage

\vfill\eject

\begin{center} {\large\bf Table 2: Classification of non-nilpotent complex structures}\vspace{5pt}

{
\begin{tabular}{|c|l|}
\hline  $\frg$ & \hskip3cm{Complex structures} \\
\hline &\\[-9pt] $\frh^-_{19}$ &
$d\omega^1=0,\quad d\omega^2 = \omega^{13} + \omega^{1\bar{3}},\quad
d\omega^3 = \pm i(\omega^{1\bar{2}} - \omega^{2\bar{1}})$
\\[1pt]
\hline &\\[-9pt] $\frh^+_{26}$ &
$d\omega^1=0,\quad d\omega^2 = \omega^{13} + \omega^{1\bar{3}},\quad
d\omega^3 = i\,
\omega^{1\bar{1}} \pm i ( \omega^{1\bar{2}} - \omega^{2\bar{1}})$ \\[1pt]
\hline
\end{tabular}}
\end{center}

\section{Fr\"olicher spectral sequence} \label{Frolicher}

\noindent
In this section we study the behaviour of the Fr\"olicher sequence for 6-nilmanifolds endowed with an invariant complex structure.
Recall that given a complex manifold~$M$, the Fr\"olicher spectral sequence $E_r^{p,q}(M)$ is the
spectral sequence associated to the double complex $(\Omega^{p,q}(M),\partial,\db)$, where
$\partial$ and $\db$ come from the well-known decomposition $d=\partial+\db$ of the exterior
differential $d$ on $M$~\cite{Fro}.

The first term $E_1(M)$ in the sequence is precisely the Dolbeault cohomology of~$M$, that is,
$E_1^{p,q}(M)\cong H_{\db}^{p,q}(M)$, and after a finite number of steps this sequence converges to
the de Rham cohomology of $M$.
More concretely, for each $r\geq 1$ there is a sequence of homomorphisms $d_r$
$$
\cdots\longrightarrow E_r^{p-r,q+r-1}(M) \stackrel{d_r}{\longrightarrow} E_r^{p,q}(M)
\stackrel{d_r}{\longrightarrow} E_r^{p+r,q-r+1}(M) \longrightarrow\cdots
$$
such that $d_r\circ d_r=0$ and $E_{r+1}^{p,q}(M)={\rm Ker}\, d_r / {\rm Im}\, d_r$.
The homomorphisms $d_r$ are induced from $\partial$.
When $r=1$ the homomorphism $d_1\colon H_{\bar\partial}^{p,q}(M) \longrightarrow H_{\bar\partial}^{p+1,q}(M)$ is
given by $d_1([\alpha_{p,q}])=[\partial\alpha_{p,q}]$, for $[\alpha_{p,q}]\in H_{\bar\partial}^{p,q}(M)$.
We will also use that
$$E_2^{p,q}(M)=\frac{\{\alpha_{p,q}\in \Omega^{p,q}(M)\,|\,\bar\partial \alpha_{p,q}=0,\, \partial\alpha_{p,q}=-\db\alpha_{p+1,q-1}\}}{\{ \db\beta_{p,q-1}+\partial\gamma_{p-1,q} \,|\, \db\gamma_{p-1,q}=0\}},
$$
and the homomorphism $d_2\colon E_{2}^{p,q}(M) \longrightarrow E_{2}^{p+2,q-1}(M)$ is given by
$d_2([\alpha_{p,q}]) = [\partial\alpha_{p+1,q-1}]$,
for $[\alpha_{p,q}]\in E_{2}^{p,q}(M)$ (see for example \cite{CFGU0} for general descriptions of $d_r$ and $E_r^{p,q}$).

Let $M=\nilm$ be a nilmanifold endowed with an invariant complex structure $J$,
and let $\frg$ be the Lie algebra of $G$. In dimension 6,
Rollenske proved in \cite[Section 4.2]{R2} that if $\frg\not\cong\frh_7$ then the natural inclusion
$$
\left(\bigwedge\phantom{\!\!}^{p,\bullet}(\frg^*),\db\right) \hookrightarrow (\Omega^{p,\bullet}(M),\db)
$$
induces an isomorphism
\begin{equation}\label{iso-dolb}
\iota\colon H_{\db}^{p,q}(\frg) \longrightarrow H_{\db}^{p,q}(M)
\end{equation}
between the Lie-algebra Dolbeault cohomology of $(\frg,J)$ and the Dolbeault cohomology of $M$.
Thus, an inductive argument~\cite[Theorem 4.2]{CFGU1bis} implies that the natural map
$\iota\colon E_r^{p,q}(\frg) \longrightarrow E_r^{p,q}(M)$ is also an isomorphism, and therefore
$E_r^{p,q}(M)\cong E_r^{p,q}(\frg)$ for any $p,q$ and any $r\geq 1$, whenever $\frg\not\cong\frh_7$
(see Remark~\ref{explic-h7} below).
Using this, in the next result we show the general behaviour of the Fr\"olicher sequence in dimension~6.

\begin{theorem}\label{Fro-general}
Let $M=\nilm$ be a $6$-dimensional nilmanifold endowed with an invariant complex structure $J$ such that the underlying Lie algebra
$\frg\not\cong \frh_7$. Then the Fr\"olicher spectral sequence $E_r^{p,q}(M,J)$ behaves as follows:
\begin{enumerate}
\item[(a)] If $\frg\cong \frh_1$, $\frh_3$, $\frh_6$,
$\frh_8$, $\frh_9$,
$\frh_{10}$, $\frh_{11}$, $\frh_{12}$ or $\frh_{19}^-$, then $E_1\cong E_{\infty}$ for any $J$.
\item[(b)] If $\frg\cong\frh_2$ or $\frh_4$, then $E_1\cong E_{\infty}$ if and only if $J$ is non-abelian; moreover,
any abelian complex structure on $\frh_2$ or $\frh_4$ satisfies $E_1\not\cong E_2\cong E_{\infty}$.
\item[(c)]
If $\frg\cong\frh_5$ and $J$ is a complex structure on $\frh_5$ given in Table 1, then:
\begin{enumerate}
\item[(c.1)]
$E_1\not\cong E_2\cong E_{\infty}$ when $J$ is complex-parallelizable;
\item[(c.2)]
$E_1\cong E_{\infty}$ if and only if $J$ is not complex-parallelizable and $\rho D\not=0$;
moreover, $E_1\not\cong E_2\cong E_{\infty}$ when $\rho D=0$.
\end{enumerate}
\item[(d)]
If $\frg\cong\frh_{16}$ or $\frh_{26}^+$, then $E_1\not\cong E_2\cong E_{\infty}$ for any $J$.
\item[(e)]
If $\frg\cong\frh_{13}$ or $\frh_{14}$, then $E_1\cong E_2\not\cong E_3 \cong E_{\infty}$ for any $J$.
\item[(f)]
If $\frg\cong\frh_{15}$ and $J$ is a complex structure on $\frh_{15}$ given in Table 1, then:
\begin{enumerate}
\item[(f.1)]
$E_1\not\cong E_2\cong E_{\infty}$, when $c=0$ and $|B-\rho|\not= 0$;
\item[(f.2)]
$E_1\cong E_2\not\cong E_3 \cong E_{\infty}$, when $\rho=1$ and $|B-1|\not= c\not= 0$;
\item[(f.3)]
$E_1\not\cong E_2\not\cong E_3 \cong E_{\infty}$, when $\rho=0$ and $|B|\not= c\not= 0$.
\end{enumerate}
\end{enumerate}
\end{theorem}

\begin{proof}
The proof is straightforward and we only give it explicitly for the case (f), that is,
$\frg\cong\frh_{15}$, because it is the most intriguing case where different non-trivial behaviours can be produced.

We will use the notation $E_r^{|k|}=\oplus_{p+q=k}E_r^{p,q}$. Since $E_{\infty}^{|k|}\cong H_{\rm dR}^{k}$,
it is clear that
$\dim E_r^{|k|} \geq b_k=\dim H^k_{dR}$ for all $k$, and the equalities hold if and only if $E_r\cong E_{\infty}$.  Recall that $b_1(\frh_{15}) = 3$, $b_2(\frh_{15}) = 5$ and $b_3(\frh_{15}) = 6$ (see \cite{S}).

For the calculation of the first term $E_1$, that is, the Dolbeault cohomology,
by the Serre duality it suffices to study the spaces $E_1^{p,q}=H_{\db}^{p,q}$ for $(p,q)=(1,0), (0,1), (2,0), (1,1), (0,2), (3,0)$ and $(2,1)$.

Let $J$ be a complex structure on $\frh_{15}$ given in Table 1.
If $J$ is abelian then $(B,c)=(0,1)$ or $(1,c)$ with $c\not=1$, therefore
\begin{equation}\label{grupos-Dolb-family2-ro0}
\begin{array}{l}
H_{\db}^{1,0}=\langle [\omega^{1}] \rangle,\quad\quad
H_{\db}^{2,0}=\langle [\omega^{12}],  \delta_0^c[\omega^{13}]\rangle, \quad \quad
H_{\db}^{3,0}=\langle [\omega^{123}] \rangle,\\[7pt]
H_{\db}^{0,1}=\langle [\omega^{\bar{1}}], [\omega^{\bar{2}}], [\omega^{\bar{3}}] \rangle,\quad\quad
H_{\db}^{0,2}=\langle [\omega^{\bar{1}\bar{2}}], [\omega^{\bar{1}\bar{3}}], [\omega^{\bar{2}\bar{3}}] \rangle,\\[7pt]
H_{\db}^{1,1}=\langle (1-\delta_0^c)[\omega^{1\bar{2}}],  [\omega^{1\bar3}], \delta_0^c[\omega^{2\bar{1}}], [B\omega^{2\bar{2}}+\omega^{3\bar{1}}], \delta_0^c[\omega^{3\bar{2}}] \rangle, \quad\ \\[7pt]
H_{\db}^{2,1}=\langle \delta_0^c[\omega^{12\bar{1}}], [\omega^{12\bar{2}}],  [\omega^{12\bar{3}}], [B\omega^{13\bar{2}}-c\omega^{23\bar{1}}], \delta_0^c[\omega^{13\bar{3}}] \rangle,
\end{array}
\end{equation}
where $\delta^c_0$ is equal to $0$ if $c\not= 0$, and equals $1$ if $c=0$.
Since $\dim E_1^{|1|}=4>3=b_1(\frh_{15})$ we get that $E_1\not\cong E_{\infty}$ for
any abelian $J$.

When $J$ is not abelian, i.e. $\rho=1$, the Dolbeault cohomology groups are
\begin{equation}\label{grupos-Dolb-family2-ro1}
\begin{array}{l}
H_{\db}^{1,0}=\langle [\omega^{1}], \delta^B_0\delta^c_0[\omega^{3}] \rangle,\quad\quad
H_{\db}^{2,0}=\langle [\omega^{12}], \delta^c_0[\omega^{13}] \rangle,\quad\quad
H_{\db}^{3,0}=\langle [\omega^{123}] \rangle,\\[7pt]
H_{\db}^{0,1}=\langle [\omega^{\bar{1}}], [\omega^{\bar{2}}] \rangle,\quad\quad
H_{\db}^{0,2}=\langle [\omega^{\bar{1}\bar{3}}], [\omega^{\bar{2}\bar{3}}] \rangle,\\[7pt]
H_{\db}^{1,1}=\langle (Bc+\delta^B_0)[\omega^{1\bar{2}}], [\omega^{1\bar{3}}+\omega^{2\bar{2}}],
[B\omega^{1\bar{3}}-\omega^{3\bar{1}}], \delta^c_0[\omega^{2\bar{1}}], \delta^c_0[\omega^{3\bar{2}}] \rangle,\\[7pt]
H_{\db}^{2,1}=\langle \delta^c_0[\omega^{12\bar{1}}], [\omega^{12\bar{2}}], [c\,\omega^{12\bar{3}}+\omega^{13\bar{2}}],
[B\omega^{12\bar{3}}+\omega^{23\bar{1}}], \delta^c_0[\omega^{13\bar{3}}+\omega^{23\bar{2}}]  \rangle,
\end{array}
\end{equation}
where $\delta^B_0$ has a similar definition as for $\delta^c_0$ above.
Notice that the coefficient $Bc+\delta^B_0$ is non-zero except for $B\not=0$ and $c=0$.
Thus, $\dim E_1^{|2|}\geq 6>5=b_2(\frh_{15})$ and so $E_1\not\cong E_{\infty}$ also for any non-abelian $J$.

In order to prove (f.1) we need to study independently the abelian and the non-abelian complex structures with $c=0$ and $B\neq \rho$ on $\frh_{15}$.  We start with the abelian ones. In this case, by Table 1 we can suppose $B=1$ and from \eqref{grupos-Dolb-family2-ro0} it follows that the dimensions of $E_1^{|2|}$ and $E_1^{|3|}$ are
$$
\dim E_1^{|2|}=9>5=b_2(\frh_{15}),\quad\quad \dim E_1^{|3|}=12>6=b_3(\frh_{15}).
$$
For the following $d_1$-homomorphisms
$
E_1^{0,1} \stackrel{d_1}{\longrightarrow} E_1^{1,1}
\stackrel{d_1}{\longrightarrow} E_1^{2,1}
\stackrel{d_1}{\longrightarrow} E_1^{3,1},$
the classes $[\omega^{\bar{3}}]$, $[\omega^{1\bar{3}}]$, $[\omega^{3\bar{2}}]$, $[\omega^{13\bar{3}}]$ have linearly independent images.  On the other hand, for
$ E_1^{0,2} \stackrel{d_1}{\longrightarrow} E_1^{1,2}
\stackrel{d_1}{\longrightarrow} E_1^{2,2}
\stackrel{d_1}{\longrightarrow} E_1^{3,2},
$
the images of the classes $[\omega^{\bar{2}\bar{3}}]$, $[\omega^{3\bar{2}\bar{3}}]$, $[\omega^{2\bar{2}\bar{3}} + \omega^{3\bar{1}\bar{3}}]$ and
$[\omega^{13\bar{2}\bar{3}}]$ are also independent.  Counting dimensions for $E^{|k|}_2$ we get that

$\dim E_2^{|1|}\leq \dim E_1^{|1|}-1=3=b_1(\frh_{15})$, \quad $\dim E_2^{|2|}\leq \dim E_1^{|2|}-4=5=b_2(\frh_{15})$,

$\dim E_2^{|3|}\leq \dim E_1^{|3|}-6=6=b_3(\frh_{15})$, \quad $\dim E_2^{|4|}\leq \dim E_1^{|4|}-4=5=b_4(\frh_{15})$,

$\dim E_2^{|5|}\leq \dim E_1^{|5|}-1=3=b_5(\frh_{15})$.

\noindent This implies that $E_2\cong E_{\infty}$ because necessarily $\dim E_2^{|k|}= b_k(\frh_{15})$ for all $k$.

If $\rho=1$ and $c=0$, then $B\neq 1$ and by \eqref{grupos-Dolb-family2-ro1} we have $\dim E_1^{|1|}=b_1(\frh_{15})+\delta_0^B$.
So $E_1^{|1|}\cong E_{\infty}^{|1|}$ when $B\neq 0$. For $B=0$, since $d_1([\omega^3])\neq 0$ and $d_1([\omega^{3\bar{1}\bar{2}\bar{3}}])\neq 0$,
we conclude that $\dim E_2^{|1|}\leq \dim E_1^{|1|}-1=3=b_1(\frh_{15})$ and $\dim E_2^{|5|}\leq \dim E_1^{|5|}-1=3=b_1(\frh_{15})$, and therefore, $E_2^{|k|}\cong E_{\infty}^{|k|}$ if $k=1$ or $k=5$.

Now, for $B\neq 1$ we have that $\dim E_1^{|2|}=8+\delta_0^B>5=b_2(\frh_{15})$,  $\dim E_1^{|3|}=12>6=b_3(\frh_{15})$,   $\dim E_1^{|4|}=8+\delta_0^B>5=b_4(\frh_{15}).$  In order to conclude that $E_2\cong E_{\infty}$ it suffices to observe that for
the following homomorphisms
$$
E_1^{1,1} \stackrel{d_1}{\longrightarrow} E_1^{2,1}
\stackrel{d_1}{\longrightarrow} E_1^{3,1}, \quad\quad
E_1^{0,2} \stackrel{d_1}{\longrightarrow} E_1^{1,2}
\stackrel{d_1}{\longrightarrow} E_1^{2,2}
$$
the classes $[\omega^{1\bar{3}} + \omega^{2\bar2}]$, $[\omega^{3\bar{2}}]$, $[\omega^{13\bar{3}} + \omega^{23\bar{2}}]$,
$[\omega^{\bar{2}\bar{3}}]$, $[\omega^{3\bar{2}\bar{3}}]$ and
$[B\omega^{2\bar{2}\bar{3}}+\omega^{3\bar{1}\bar{3}}]$ have linearly independent images.

For case (f.2), we consider $\rho=1$ and $|B-1|\neq c\neq 0$.
As $\dim E_1^{|1|}=3=b_1(\frh_{15})$, we get that $E_1^{|1|}\cong E_{\infty}^{|1|}$.
Now, for the map $E_2^{0,2} \stackrel{d_2}{\longrightarrow} E_2^{2,1}$ we have
$d_2([\omega^{\bar{2}\bar{3}}]) = \left[\partial\left(\omega^{2\bar3} + \frac{1-\bar B}{c}\,\omega^{3\bar2}\right)\right] =
\frac{|B-1|^2 - c^2}{c}\,[\omega^{12\bar2}]\not=0$,
because $\omega^{12\bar2}\not= \db\beta_{2,0}+\partial\gamma_{1,1}$ for any $\beta_{2,0}$ and any $\db$-closed $\gamma_{1,1}$.
Hence,
$$b_2(\frh_{15})\leq \dim E_3^{|2|}\leq \dim E_2^{|2|}-1\leq \dim E_1^{|2|}-1=6-1=5=b_2(\frh_{15})$$ and we conclude that $E_{\infty}^{|2|}\cong E_3^{|2|}\not\cong E_2^{|2|} \cong E_1^{|2|}$.

Similarly,  $d_2\colon E_2^{1,2} \longrightarrow E_2^{3,1}$ is non-zero
(for instance, $d_2([\omega^{3\bar{1}\bar{3}} + B\omega^{2\bar{2}\bar{3}}])\not=0$).
Thus,
$$b_3(\frh_{15})\leq \dim E_3^{|3|}\leq \dim E_2^{|3|}-2\leq \dim E_1^{|3|}-2=8-2=6=b_3(\frh_{15})$$
and we conclude that $E_{\infty}^{|3|}\cong E_3^{|3|}\not\cong E_2^{|3|} \cong E_1^{|3|}$.  By the same argument
$$b_4(\frh_{15})\leq \dim E_3^{|4|}\leq \dim E_2^{|4|}-1\leq \dim E_1^{|4|}-1=6-1=5=b_4(\frh_{15})$$
and therefore $E_{\infty}^{|4|}\cong E_3^{|4|}\not\cong E_2^{|4|} \cong E_1^{|4|}$.
Summing up all the information, we conclude that $E_1\cong E_2\not\cong E_3 \cong E_{\infty}$ in case (f.2).

For the last case (f.3), we first observe that
$d_1([\omega^{\bar{3}}])=-c[\omega^{1\bar{2}}]-\bar{B}[\omega^{2\bar{1}}]$. Since this class is zero if and only if $c\,\omega^{1\bar{2}}+\bar{B}\omega^{2\bar{1}}\in \bar\partial(\bigwedge^{1,0})=\langle \omega^{1\bar1}, B\omega^{1\bar2}+c\,\omega^{2\bar1}\rangle$, i.e. $|B|=c$,
the map $d_1\colon E_1^{0,1} \longrightarrow E_1^{1,1}$ is non-zero. Therefore,
$\dim E_2^{|1|}\leq \dim E_1^{|1|}-1=3$, i.e. $E_1^{|1|}\not\cong E_2^{|1|}\cong E_{\infty}^{|1|}$.  Moreover, since $d_2([\omega^{\bar2\bar3}])\neq 0$, we deduce that
$$b_2(\frh_{15})\leq \dim E_3^{|2|}\leq \dim E_2^{|2|}-1\leq \dim E_1^{|2|}-2=7-2=5=b_2(\frh_{15}),$$
so $E_{\infty}^{|2|}\cong E_3^{|2|}\not\cong E_2^{|2|}\not \cong E_1^{|2|}$.
Analogously, $d_2([\omega^{3\bar{1}\bar{3}} + B\omega^{2\bar{2}\bar{3}}])\neq 0$, which implies
$$b_3(\frh_{15})\leq \dim E_3^{|3|}\leq \dim E_2^{|3|}-2\leq \dim E_1^{|3|}-2=8-2=6=b_3(\frh_{15}),$$ and we conclude that $E_{\infty}^{|3|}\cong E_3^{|3|}\not\cong E_2^{|3|} \cong E_1^{|3|}$.  We also have
$$b_4(\frh_{15})\leq \dim E_3^{|4|}\leq \dim E_2^{|4|}-1\leq \dim E_1^{|4|}-2=7-2=5=b_4(\frh_{15}),$$
and therefore $E_{\infty}^{|4|}\cong E_3^{|4|}\not\cong E_2^{|4|} \not\cong E_1^{|4|}$.
Consequently, $E_1\not\cong E_2\not\cong E_3\cong E_{\infty}$ in case (f.3).
\end{proof}

\begin{remark}\label{explic-h7}
{\rm
Let $(M = \nilm, J)$ be a 6-dimensional nilmanifold endowed
with an invariant complex structure $J$ and suppose that $\mathfrak{g} = \mathfrak{h}_7$.
In \cite[Theorem 4.4]{R2} it is proved that
there is a dense subset of the space of all invariant
complex structures for which the complex nilmanifold admits the structure of principal
holomorphic bundle of elliptic curves over a Kodaira surface, but
this is not true for all complex structures. In fact, the invariant complex structure $J$ may not be compatible with the
lattice $\Gamma$ (see \cite[Example 1.14]{R2}), so one cannot ensure the existence of the isomorphism
\eqref{iso-dolb}, and hence of a canonical isomorphism between $E_r^{p,q}(\mathfrak{g},J)$
and $E_r^{p,q}(M,J)$, for any invariant $J$ on the nilmanifold $M$.
However, notice that up to equivalence there is only one complex structure on $\mathfrak{h}_7$ and it can be proved that it satisfies that
the sequence degenerates at the first step, i.e.
$E_1(\mathfrak{h}_7)\cong E_{\infty}(\mathfrak{h}_7)$.
}
\end{remark}

In \cite{AT} the authors posed the following problem: to construct a compact complex manifold such that
$E_1\cong E_\infty$ and $h_{\db}^{p,q}=h_{\db}^{q,p}$ for every $p,q\in \mathbb{N}$ but for which
the $\partial\db$-lemma does not hold.
Since nilmanifolds do not satisfy the $\partial\db$-lemma, unless they are complex tori, the following
result provides a solution.

\begin{proposition}\label{h6-question}
Let $J$ be any invariant complex structure on a nilmanifold $M$ with underlying Lie algebra
isomorphic to $\frh_6$. Then $E_1(M)\cong E_\infty(M)$ and the Hodge numbers satisfy
$$
\begin{array}{c}
 h_{\db}^{0,0}(M)=1,\\[4pt]
 h_{\db}^{1,0}(M)=2,\quad\ h_{\db}^{0,1}(M)=2,\\[4pt]
 h_{\db}^{2,0}(M)=2,\quad\ h_{\db}^{1,1}(M)=5,\quad\ h_{\db}^{0,2}(M)=2,\\[4pt]
 h_{\db}^{3,0}(M)=1,\quad\ h_{\db}^{2,1}(M)=5,\quad\ h_{\db}^{1,2}(M)=5, \quad\ h_{\db}^{0,3}(M)=1,\\[4pt]
 h_{\db}^{3,1}(M)=2,\quad\ h_{\db}^{2,2}(M)=5,\quad\ h_{\db}^{1,3}(M)=2,\\[4pt]
 h_{\db}^{3,2}(M)=2,\quad\ h_{\db}^{2,3}(M)=2,\\[4pt]
 h_{\db}^{3,3}(M)=1.
\end{array}
$$
\end{proposition}

\begin{proof}
Any complex structure $J$ on $\frh_6$ is equivalent to the complex structure given in Table 1, that is,
$\rho=\lambda=1$ and $D=0$.
Its Dolbeault cohomology groups $H_{\db}^{p,q}$ for $(p,q)=(1,0), (0,1), (2,0), (1,1), (0,2), (3,0)$ and $(2,1)$ are
$$
\begin{array}{l}
H_{\db}^{1,0}=\langle [\omega^{1}], [\omega^{2}] \rangle,\quad\quad
H_{\db}^{2,0}=\langle [\omega^{12}], [\omega^{13}] \rangle,\quad\quad
H_{\db}^{3,0}=\langle [\omega^{123}] \rangle,\\[7pt]
H_{\db}^{0,1}=\langle [\omega^{\bar{1}}], [\omega^{\bar{2}}] \rangle,\quad\quad
H_{\db}^{0,2}=\langle [\omega^{\bar{1}\bar{3}}], [\omega^{\bar{2}\bar{3}}] \rangle,\\[7pt]
H_{\db}^{1,1}=\langle [\omega^{1\bar{2}}], [\omega^{2\bar{1}}], [\omega^{2\bar{2}}],
[\omega^{1\bar{3}}+\omega^{3\bar{2}}], [\omega^{3\bar{1}}+\omega^{3\bar{2}}] \rangle,\\[7pt]
H_{\db}^{2,1}=\langle [\omega^{12\bar{2}}], [\omega^{13\bar{1}}],
[\omega^{12\bar{3}}+\omega^{23\bar{1}}], [\omega^{12\bar{3}}-\omega^{23\bar{2}}], [\omega^{13\bar{2}}] \rangle.
\end{array}
$$
By Serre duality we get the above Hodge diamond which is symmetric.
Moreover,
$$\dim E_1^{|1|}=4=b_1(\frh_6),\quad \dim E_1^{|2|}=9=b_2(\frh_6),\quad \dim E_1^{|3|}=12=b_3(\frh_6),$$
so the Fr\"olicher spectral sequence degenerates at the first step.
\end{proof}

The following result shows that there are many complex nilmanifolds for which the Fr\"olicher spectral sequence is stable under small
deformations of the complex structure.

\begin{proposition}\label{stable-Frol-sequence}
Let $M=\nilm$ be a $6$-dimensional nilmanifold endowed with an invariant complex structure $J$,
and let $\frg$ be the Lie algebra of $G$.
If $\frg\cong \frh_1,\frh_3,\frh_6, \frh_8,\frh_9, \frh_{10}, \frh_{11},\frh_{12},\frh_{13}, \frh_{14},\frh_{16},\frh_{19}^-$ or $\frh_{26}^+$,
then $\dim E_r^{p,q}(M,J)$ is stable under small deformations of $J$ for any $p,q$ and any $r\geq 1$.
\end{proposition}

\begin{proof}
By~\cite[Theorem 2.6]{R1}, all small deformations of the complex structure $J$ are again invariant complex structures.
Proceeding as in the proof of Theorem~\ref{Fro-general}, it can be proved that if
$\frg\not\cong \frh_2, \frh_4, \frh_5$ or $\frh_{15}$, then
$\dim E_r^{p,q}(M)$ does not depend on the invariant complex structure on $M$ for any $p,q$ and any $r\geq 1$, so it is stable under small
deformations of $J$.
\end{proof}

\begin{remark}\label{abel-no-abel}
{\rm
The 6-dimensional nilmanifolds with underlying Lie algebra isomorphic to $\frh_2$, $\frh_4$, $\frh_5$ or $\frh_{15}$
are the only ones that
have both abelian and non-abelian complex structures (see Table~1).
More generally, let $M$ be a $2n$-dimensional nilmanifold, $\frg$ the underlying Lie algebra, $J$ an abelian complex structure and
$J'$ a non-abelian invariant complex structure on $M$.
It is well known that $J$ is abelian if and only if there is a basis $\{\omega^1,\ldots,\omega^n\}$
of invariant forms of type (1,0) satisfying $\partial\omega^j=0$ for $1\leq j\leq n$; therefore, by \cite{CF} one has that
$h_{\bar\partial}^{0,1}(M,J)=n$ because \eqref{iso-dolb} holds for abelian structures.
However, for $J'$ we have $\dim H_{\bar\partial}^{0,1}(\frg,J')<n$ and, if an isomorphism like
\eqref{iso-dolb} holds,
then the Hodge number satisfies $h_{\bar\partial}^{0,1}(M,J')<n$. Thus, the existence of $J$ and $J'$ on a nilmanifold $M$
might lead to the non-stability of $\dim E_1^{0,1}$ under small deformations.
A natural question arises in this context: is the Fr\"olicher spectral sequence stable under small deformations
if and only if the nilmanifold does not admit both abelian and non-abelian complex structures?
Proposition~\ref{stable-Frol-sequence} above gives an affirmative answer for $n=3$.
}
\end{remark}

Next we provide some examples of explicit families of complex structures on nilmanifolds
corresponding to $\frh_5$ and $\frh_{15}$ along which the Fr\"olicher sequence varies.
In Corollaries~\ref{nonclosedFro} and~\ref{nonopenE2} below, further properties of
the Fr\"olicher spectral sequence on nilmanifolds are shown.

\begin{example}\label{example-h5}
{\rm
Let $J$ be a non complex-parallelizable and non-abelian complex structure on $\frh_5$ given in Table 1
with non-degenerate Fr\"olicher sequence, i.e. $E_1\ncong E_{\infty}$ for $J$.  We will construct a family of complex structures $J_t$ by deforming the previous one, i.e. $J_0=J$,  such that the Fr\"olicher spectral sequence degenerates at the first step for any $t\neq 0$.
According to Theorem~\ref{Fro-general}, $J$ has complex structure equations of the form
$$d\omega^1=d\omega^2=0,\quad d\omega^3=\omega^{12}+\omega^{1\bar1} + \lambda\,\omega^{1\bar2},$$
for some non-negative $\lambda\neq 1$, where $\{\omega^1, \omega^2, \omega^3\}$ is a $(1,0)$-basis for $J$.
With respect to the real basis $\{e^1,\ldots,e^6\}$ given by
$$
e^1 + i\,e^2=\omega^1,\quad
\frac{1}{1+\lambda}(e^3-e^1) +
\frac{i}{1-\lambda}(e^2+e^4)=\omega^2,\quad
e^5 + i e^6=\omega^3,
$$
the complex structure $J$ expresses as
$$
\begin{array}{lll}
J e^1=-e^2,\quad & J e^3=-\frac{2}{1-\lambda}\,e^2 - \frac{1+\lambda}{1-\lambda}\,e^4,\quad & J e^5=-e^6, \\[6pt]
J e^2=e^1,\quad &
J e^4=-\frac{2}{1+\lambda}\,e^1 + \frac{1-\lambda}{1+\lambda}\,e^3,\quad & J e^6=e^5.
\end{array}
$$
For any $t\in[0,\frac12)$, consider the complex structure $J_t$ given by
$$
\begin{array}{l}
J_t e^1=\frac{4d(1-\lambda)}{\alpha^2}\,e^1 - \frac{1-\lambda^2}{\alpha}\,e^2 - \frac{2d(1-\lambda)^2}{\alpha^2}\,e^3 + \frac{8d^2(1-\lambda)}{\alpha^3}\,e^4,\\[10pt]
J_t e^2=\frac{1-\lambda^2}{\alpha}\,e^1 +  \frac{2d(1-\lambda^2)}{\alpha^2}\,e^4,\\[10pt]
J_t e^3=-\frac{2d}{(1-\lambda)^2}\,e^1 - \frac{2\alpha}{(1-\lambda^2)(1-\lambda)}\,e^2 - \frac{(1+\lambda)^2}{\alpha}\,e^4,\\[10pt]
J_t e^4=-\frac{2(1-\lambda)}{\alpha}\,e^1 + \frac{2d}{1-\lambda^2}\,e^2 + \frac{(1-\lambda)^2}{\alpha}\,e^3 - \frac{4d(1-\lambda)}{\alpha^2}\,e^4,\\[10pt]
J_t e^5=\frac{2d}{1-\lambda^2}\,e^5 -\frac{4d^2+(1-\lambda^2)^2}{\alpha(1-\lambda^2)}\,e^6,\\[10pt]
J_t e^6=\frac{\alpha}{1-\lambda^2}\,e^5 - \frac{2d}{1-\lambda^2}\,e^6,
\end{array}
$$
where $\alpha=\sqrt{(1-\lambda^2)^2-4d^2}$, and
$$
d(t,\lambda)=\begin{cases}
\begin{array}{ll}
t,& \text{if }\lambda=0,\\
t\lambda^2/4,& \text{if }\lambda^2\in(0,1/2),\\
t(1-\lambda^2)/4,& \text{if }\lambda^2\in[1/2,1),\\
-t (1-\lambda^2)/4,& \text{if }\lambda^2>1.
\end{array}
\end{cases}
$$
Notice that $J_0=J$. Now, the  forms 
$$
\begin{array}{l}
\omega_t^1=\frac{1-\lambda^2}{\alpha}\,e^1 + \frac{2d(1-\lambda^2)}{\alpha^2}\,e^4 + i\,e^2,\\[7pt]
\omega_t^2=\frac{1-\lambda}{\alpha}(e^3-e^1) -
\frac{2d(1-\lambda)}{\alpha^2} e^4 +
\frac{i}{1-\lambda} \left( \frac{2d}{\alpha} e^1
+e^2+\frac{(1-\lambda^2)^2}{\alpha^2} e^4 \right),\\[7pt]
\omega_t^3=e^5-\frac{2d}{\alpha} e^6 + i\frac{1-\lambda^2}{\alpha}
e^6,
\end{array}
$$
satisfy $J_t\omega_t^k = i\,\omega_t^k$ for $k=1,2,3$, i.e. $\{\omega^1_t, \omega^2_t, \omega^3_t\}$ is a basis of type $(1,0)$ for $J_t$.  Furthermore,  with respect to this basis the complex structure equations are $$d\omega_t^1=d\omega_t^2=0,\quad d\omega_t^3=\omega_t^{12}+\omega_t^{1\bar1} + \lambda\,\omega_t^{1\bar2}+D\,\omega_t^{2\bar2},$$
with $D=i\,d(t,\lambda)$.
According to Theorem~\ref{Fro-general}, the Fr\"olicher spectral sequence degenerates if and only if $D\neq0$, i.e.
if and only if $t>0$. In conclusion, $J$ can be deformed into a non-abelian complex structure with $E_1\cong E_{\infty}$.
}
\end{example}

\begin{corollary}\label{def-h5-Fro}
Let $M=\nilm$ be the nilmanifold underlying the Iwasawa manifold, i.e. $\frg\cong \frh_5$.
Let $J$ be a non complex-parallelizable and non-abelian complex structure on $M$ given in Table 1
with $E_1\not\cong E_{\infty}$.
Then, $J$ can be deformed into an invariant complex structure with degenerate Fr\"olicher spectral sequence.
\end{corollary}

The Lie algebra $\frh_{15}$ has a rich complex geometry with respect to the Fr\"olicher sequence and in the next example
we construct a family $J_t$ along which the three cases in (f) of Theorem~\ref{Fro-general} are realized.

\begin{example}\label{example-h15}
{\rm
On $\frh_{15}$, let us consider the real basis $\{e^1,\ldots, e^6\}$ given in Theorem~\ref{clasif-complex}  and the following family of complex structures
\begin{eqnarray*}
J_te^1&=&-\sqrt{\frac{3(3-\sin t)(7+3\sin t)}{(5+\sin t)(11-\sin t)}}\,e^2, \\[6pt]
J_te^3&=&\sqrt{\frac{3(3-\sin t)(11-\sin t)}{(5+\sin t)(7+3\sin t)}}\,e^4, \\[6pt]
J_te^5&=&-\sqrt{\frac{(11-\sin t)(7+3\sin t)}{3(3-\sin t)(5+\sin t)}}\,e^6,
\end{eqnarray*}
where $t\in \mathbb{R}$.
Let
\begin{eqnarray*}
4\,\omega^1_t \!\!\!&=&\!\!\! \sqrt{(11-\sin t)(5+\sin t)}\,e^1 + i\, \sqrt{3(3-\sin t)(7+3\sin t)}\,e^2,\\[6pt]
8\, \omega^2_t \!\!\!&=&\!\!\! (5+\sin t)(7+3\sin t)\,e^3 - i\, \sqrt{3(5+\sin t)(3-\sin t)(11-\sin t)(7+3\sin t)}\,e^4,
\end{eqnarray*}
and
\begin{eqnarray*}
128\, \omega^3_t=(5+\sin t)(7+3\sin t)\!\! \!\!\!&\!\!\!&\!\!\! \left[3(3-\sin t)\sqrt{(11-\sin t)(5+\sin t)}\,e^5 \right. \\[4pt]
\!\!\!&\!\!\!&\!\!\! \left. \ + \,i\, (11-\sin t)\sqrt{3(3-\sin t)(7+3\sin t)}\,e^6 \right].
\end{eqnarray*}
Then, $\{\omega^1_t,\omega^2_t,\omega^3_t\}$ is a $(1,0)$-basis for $J_t$
satisfying
$$
d\omega^1_t=0,\quad
d\omega^2_t=\omega_t^{1\bar1},\quad
d\omega^3_t=\frac{1-\sin t}{2}\,\omega_t^{12} + 2\,\omega_t^{1\bar2} + \frac{1+\sin t}{4}\,\omega_t^{2\bar1}.
$$

If $\sin t =1$, 
the coefficient of $\omega_t^{12}$ vanishes and therefore $J_t$ is an abelian complex structure that is equivalent to the one given by  $(\rho, B_t, c_t) =(0,1,\frac14)$ (see Lemma~\ref{nilp-up-to-equiv}).  If $\sin t \neq 1$, then we can normalize the coefficient of $\omega_t^{12}$ and the complex structure equations can be written in form \eqref{reduced-noabel-nilp} as
$$
d\omega_t^1=0,\quad
d\omega_t^2=\omega^{1\bar1},\quad
d\omega_t^3=\omega_t^{12} + \frac{4}{1-\sin t}\,\omega_t^{1\bar2} + \frac{1+\sin t}{2(1-\sin t)}\,\omega_t^{2\bar1},
$$ i.e. they are determined by the triple $(\rho, B_t, c_t) = \left(1, \frac{4}{1-\sin t}, \frac{1+\sin t}{2(1-\sin t)} \right)$.  Now, concerning the Fr\"olicher spectral sequence for the family $\{J_t\}_{t\in \mathbb{R}}$, by Theorem~\ref{Fro-general}~(f) we get
\begin{itemize}
\item If $\sin t=1$, then $(\rho_t, B_t, c_t) = (0, 1, \frac14)$ and therefore $E_1\not\cong E_2\not\cong E_3 \cong E_{\infty}$.
\item If $\sin t=-1$, then $(\rho_t, B_t, c_t) = (1, 2, 0)$ and $E_1\not\cong E_2\cong E_{\infty}$.
\item If $|\sin t|\neq 1$,   $E_1\cong E_2\not\cong E_3 \cong E_{\infty}$.
\end{itemize}
}\end{example}




As a consequence of this example, in the following result we show that for $r\geq 2$ the dimension of the term $E_r^{p,q}(J_t)$
in general is neither upper nor lower semi-continuous function of~$t$.
This is in deep contrast with the case $r=1$, as it is well known the upper semicontinuity
of the Hodge numbers $\dim H^{p,q}_{\db}(J_t)$ with respect to $t$ along a deformation.

\begin{corollary}\label{no-semi-continuous}
Let $M$ be a nilmanifold with underlying Lie algebra~$\frh_{15}$
endowed with the invariant complex structures $J_t$ given in Example~\ref{example-h15}.
Then,
$$
\dim E_2^{0,2}(J_{\frac{\pi}{2}})=3 > 2=\dim E_2^{0,2}(J_{t}),
\quad
\dim E_2^{1,1}(J_{\frac{\pi}{2}})=2 < 3=\dim E_2^{1,1}(J_{t}),
$$
and
$$
\dim E_3^{0,2}(J_{\frac{\pi}{2}})=2 > 1=\dim E_3^{0,2}(J_{t}),
\quad
\dim E_3^{1,1}(J_{\frac{\pi}{2}})=2 < 3=\dim E_3^{1,1}(J_{t}),
$$
for any $t\in(\frac{\pi}{2},\frac{3\pi}{2})$. Therefore, the dimensions of the terms $E_2^{1,1}(J_t)$ and $E_3^{1,1}(J_t)$
are not upper semi-continuous functions of $t$, and the dimensions of the terms $E_2^{0,2}(J_t)$ and $E_3^{0,2}(J_t)$
are not lower semi-continuous functions of $t$.
\end{corollary}

\begin{proof}
It follows directly from the proof of Theorem~\ref{Fro-general} taking into account that for $t=\frac{\pi}{2}$ the complex structure lies in case (f.3) and for any $t\in(\frac{\pi}{2},\frac{3\pi}{2})$ the structures $J_t$ lie in case (f.2).
\end{proof}

\section{Strongly Gauduchon and balanced Hermitian metrics} \label{sG}

\noindent
Let $(M,J)$ be a complex manifold of complex dimension $n$.
A Hermitian metric $g$ on $(M,J)$ can be described by means of a positive definite smooth form $\Omega$ on $M$
of bidegree $(1,1)$ with respect to $J$.
We will use this approach in what follows and we will refer to $\Omega$ as a Hermitian structure or as a Hermitian metric indistinctly.

A Hermitian structure $\Omega$ is \emph{strongly Gauduchon} (\emph{sG} for short)
if $\partial\Omega^{n-1}$ is $\db$-exact~\cite{Pop0,Pop1}. In particular, any balanced Hermitian structure (i.e. $d\Omega^{n-1}=0$)
is sG, and any sG metric is a {\em Gauduchon} metric \cite{Gau}, that is, $\Omega^{n-1}$ is
$\partial \db$-closed or equivalently the Lee form is co-closed.

Next we suppose that $(M=\nilm, J)$ is a nilmanifold endowed with an invariant complex structure.
It is proved in \cite{FG} that $(M=\nilm, J)$ has a balanced metric if and only if it has an invariant one.
Moreover, by using the symmetrization process given in \cite{Be} (see also \cite{FG}, \cite{U} and \cite[Proposition 3.2]{UV2})
one easily arrives at:

\begin{proposition}\label{sym-proc-sG}
$(M=\nilm,J)$ has an sG metric if and only if it has an invariant one.
\end{proposition}

Therefore, the existence of sG metrics on $(M=\nilm,J)$ is reduced to the existence at the Lie algebra level $\frg$ of $G$.

\begin{corollary}\label{abelian-sG}
Let $\Omega$ be an invariant Hermitian structure on $(M=\nilm,J)$.
If $J$ is abelian, then $\Omega$ is sG if and only if it is balanced.
\end{corollary}

\begin{proof}
Let $\frg$ be the Lie algebra of $G$.
First we prove that $\bar \partial (\bigwedge^{n, k}(\frg^{*})) = 0$ for every $1\leq k\leq n$.
Let us consider a decomposable form $\alpha\in\bigwedge^{n,k}(\frg^{*})$ given by $\alpha = \beta\wedge\gamma$,
where $\beta\in\bigwedge^{n,0}(\frg^{*})$ and $\gamma\in\bigwedge^{0,k}(\frg^{*})$.
Since $\frg$ is nilpotent and $J$ is abelian, one has that $d\beta=0$ and $d\gamma\in \bigwedge^{1,k}(\frg^{*})$,
so in particular $\beta$ and $\gamma$ are $\bar\partial$-closed.
Hence,
$$
\bar\partial \alpha = (\bar\partial \beta)\wedge\gamma + (-1)^{n} \beta\wedge(\bar\partial \gamma)=0.
$$
Now, the statement in the corollary follows directly from Proposition~\ref{sym-proc-sG} and from the previous property for $k=n-2$,
i.e. $\db(\bigwedge^{n,n-2}(\frg^*))=0$.
\end{proof}

From now on we consider $n=3$.

\begin{proposition}\label{dim6-sG}
Let $M=\nilm$ be a $6$-dimensional nilmanifold endowed with an invariant complex structure $J$.
There exists an sG metric on $(M=\nilm,J)$ if and only if the Lie algebra $\frg$ of $G$ is isomorphic to $\frh_1,\ldots,\frh_6$ or $\frh_{19}^-$.
\end{proposition}

\begin{proof}
By Proposition~\ref{sym-proc-sG} it suffices to study the invariant case. By \cite{U}, the fundamental 2-form of any $J$-Hermitian metric is given~by
\begin{equation}\label{F-nonnilp}
2\,\Omega=i\,(r^2\omega^{1\bar1}+s^2\omega^{2\bar2}+t^2\omega^{3\bar3})+u\omega^{1\bar2}-\bar
u\omega^{2\bar1}+v\omega^{2\bar3}-\bar
v\omega^{3\bar2}+z\omega^{1\bar3}-\bar
z\omega^{3\bar1},
\end{equation}
where coefficients $r^2,\,s^2,\,t^2$ are non-zero real numbers and $u,\, v,\, z\in\C$ satisfy $r^2s^2>|u|^2$, $s^2t^2>|v|^2$, $r^2t^2>|z|^2$ and $r^2s^2t^2 + 2\Real(i\bar u\bar v z)>t^2|u|^2 + r^2|v|^2 + s^2|z|^2$.

Let us start with the non-nilpotent case. From \eqref{eq-no-nilpot}
\begin{eqnarray*}
2\partial\Omega&=&(i\epsilon v\mp iz)\omega^{12\bar1} \mp iv\,\omega^{12\bar2} + (u-\bar u - \epsilon\,t^2)\omega^{13\bar1} + (is^2\pm t^2)\omega^{13\bar2} +\\ && v\,\omega^{13\bar3} + (is^2\mp t^2)\omega^{23\bar1}
\end{eqnarray*}
and therefore
$$
4\partial\Omega\wedge\Omega = \left(i\epsilon(s^2t^2-|v|^2)\pm (t^2u+t^2\bar u + iv\bar z - i\bar v z)\right)\,\omega^{123\bar 1\bar 2} + \left(uv-is^2z\right)\,\omega^{123\bar 1\bar 3}.
$$
Direct computations show that  $\bar\partial (\bigwedge^{3,1}(\frg^*))=\langle \omega^{123\bar 1\bar 3}\rangle$.  If the Hermitian structure $(J,\Omega)$ is sG then
$$
\mp i\epsilon(s^2t^2-|v|^2)=t^2(u+\bar u) + iv\bar z - i\bar v z.
$$
Since the left-hand side is purely imaginary and the right-hand side is real,
we get that $\epsilon=0$ and therefore $\frg\cong \frh_{19}^-$.

For the nilpotent case, let us consider the general complex equations~\eqref{nilpotentJ}. Now, the fundamental 2-form of any $J$-Hermitian metric is given also by \eqref{F-nonnilp}.  Using  \cite[Lemma 17 and Proposition 25]{U}, we get
\begin{eqnarray*}
4\partial\Omega\wedge \Omega \!\!\!&\!\!\!=\!\!\!&\!\!\!\left((1-\epsilon)\bar A(s^2t^2-|v|^2) + \bar B(it^2u + \bar vz) - \bar C(it^2\bar u - v\bar z)\right. \\ &&+ \left.(1-\epsilon)\bar D (r^2t^2-|z|^2)\right)\,\omega^{123\bar1\bar2} -\epsilon(s^2t^2-|v|^2)\,\omega^{123\bar1\bar3}.
\end{eqnarray*}
It is straightforward to verify that  $\bar\partial (\bigwedge^{3,1}(\frg^*))=\langle \rho\,\omega^{123\bar 1\bar 2}\rangle$, and therefore, if the Hermitian structure $(J,\Omega)$ is sG then $\epsilon=0$, i.e. $\frg\cong\frh_i$ for $i=1,\ldots, 6$.  Moreover, if in addition $\rho=1$, then any $J$-Hermitian structure is sG.

In conclusion, if there exists an sG metric then $\frg\cong \frh_1,\ldots,\frh_6$ or $\frh_{19}^-$.
The converse follows directly from \cite[Theorem 26]{U} because these Lie algebras admit balanced Hermitian metrics.
\end{proof}

\begin{remark}\label{anysG}
{\rm
From the proof of the previous proposition it follows that on $\frh_2,\frh_4,\frh_5$ and $\frh_6$, if $J$ is a non-abelian nilpotent complex structure then any invariant $J$-Hermitian metric is sG.
This is in contrast with $\frh_{19}^-$, where for any complex structure the space of balanced metrics is strictly contained in the
space of sG metrics, and moreover there are Hermitian metrics which are not sG. For instance, consider a Hermitian metric
on $\frh_{19}^-$ given by
$$
\Omega=\frac{i}{2}\,\omega^{1\bar1}+(u^2+z^2+1)i\,\omega^{2\bar2}+(u^2+z^2+1)i\,\omega^{3\bar3}+\frac{u}{2}(\omega^{1\bar2}-\omega^{2\bar1})+
\frac{z}{2}(\omega^{1\bar3}-\omega^{3\bar1}),
$$
that is, in \eqref{F-nonnilp} we take $r=1$, $v=0$, $u$ and $z$ real and $s^2=t^2=2(u^2+z^2+1)$:
\begin{itemize}
\item if $u=z=0$ then the metric is balanced;
\item if $u=0$ and $z\not=0$ then the metric is sG but not balanced;
\item if $u\not=0$ then the metric is not sG.
\end{itemize}
Notice that this indicates a contrast between the sG and SKT geometries, since by~\cite{FPS} the existence of an SKT structure
on a 6-dimensional nilpotent Lie algebra depends only on the complex structure.
}
\end{remark}

There exist compact complex manifolds having sG metrics
but not admitting any balanced metric~\cite[Theorem 1.8]{Pop2}.
Next we show the general situation for nilmanifolds in dimension~6.

\begin{proposition}\label{examples-sG-nobalanced}
Let $M=\nilm$ be a $6$-dimensional nilmanifold with an invariant complex structure $J$ such that
$(M=\nilm,J)$ does not admit balanced metrics. If $(M=\nilm,J)$ has sG metric, then
$J$ is non-abelian nilpotent and $\frg$ is isomorphic to $\frh_2$, $\frh_4$ or $\frh_5$.
Moreover, according to the classification in Table 1, such a $J$ is given by:
$\Real D+(\Imag D)^2\geq \frac14$ on $\frh_2$;  $\Real D\geq \frac14$ on $\frh_4$;
and $\lambda=0,\Imag D\not=0$ or $\lambda=\Imag D=0,\Real D\geq 0$ on $\frh_5$.
\end{proposition}

\begin{proof}
Any complex structure on $\frh_6$ or $\frh_{19}^-$ admits balanced metrics.
From \cite{UV2} we have that only $\frh_3$ and $\frh_5$ have abelian complex structures $J$
admitting balanced metric. In fact, any such $J$ on $\frh_5$ admits balanced Hermitian
metrics, whereas for $\frh_3$ the complex structure must be equivalent to the choice of $(-)$-sign in Table 1.
From Corollary~\ref{abelian-sG}, it remains to study the non-abelian nilpotent complex structures $J$
on $\frh_2$, $\frh_4$ and $\frh_5$. Since any such $J$ admits sG metrics by Remark~\ref{anysG},
next we show which of them do not admit balanced metrics.

In the three cases the complex equations are of the form
\begin{equation}\label{J-nilp}
d\omega^1 = d\omega^2 = 0,\quad\ \   d\omega^3 =
\omega^{12} + \omega^{1\bar{1}} + \lambda\, \omega^{1\bar{2}} +
D\, \omega^{2\bar{2}}.
\end{equation}
A similar argument as in the proof of \cite[Proposition 2.3]{UV2} shows that, up to equivalence,
the fundamental 2-form of any $J$-Hermitian metric is given~by
$$
2\,\Omega=i\,(\omega^{1\bar1}+s^2\,\omega^{2\bar2}+t^2\,\omega^{3\bar3})+u\,\omega^{1\bar2}-\bar
u\,\omega^{2\bar1},
$$
where $s^2>|u|^2$ and $t^2>0$.

If $D=x+i y$ and $u=u_1+i u_2$, the balanced condition is
\begin{equation}\label{balanced-condition-nilp}
s^2+x+i y = u_2 \lambda  + i u_1 \lambda.
\end{equation}
We distinguish several cases depending on the values of $\lambda$.

If $\lambda\not=0$ then $\Omega$ is balanced if and only if $u_1=y/\lambda$ and $u_2=(s^2+x)/\lambda$.
The condition $s^2>|u|^2$ is equivalent to $s^4+(2x-\lambda^2)s^2+x^2+y^2 < 0$
and it is easy to see that a non-zero $s$ satisfying this condition exists if and only if  the discriminant of the previous equation as a second degree equation in $s^2$ is positive, i.e.
\begin{equation}\label{lambda-balanced}
\lambda^4-4x\lambda^2-4y^2 >0.
\end{equation}
According to Table 1, non-abelian complex structures on $\frh_2$ have $\lambda=1$.  In this case \eqref{lambda-balanced} reads as $x+y^2<1/4$, which means that any $J$ such that $x+y^2\geq \frac14$ has no balanced metrics.
Similarly, for $\frh_4$ any $J$ such that $x\geq \frac14$ does not admit balanced metric.

For $\frh_5$ and $\lambda\not=0$ we have that $x=0$ by Table 1. Thus, there is no balanced metrics
if and only if $\lambda^4\leq 4y^2$. Since $y\geq 0$, this is equivalent to $\lambda^2\leq 2y$.
However, none of the three cases detailed in Table 1 verifies that $\lambda^2\leq 2y$, and therefore any complex structure on $\frh_5$ with $\lambda\neq 0$ admits balanced metrics.

Finally, in the case $\lambda=0$ on $\frh_5$ we get that the balanced condition~\eqref{balanced-condition-nilp}
reduces to $y=0$ and $s^2=-x>0$.
From Table 1 we have that $0<1+4x$, i.e. $x\in (-\frac14,\infty)$. Therefore, if $y\not=0$ or $y=0, x\geq 0$ then
there are no balanced metrics.
\end{proof}

As pointed out by Popovici~\cite{Pop2}, the degeneration of the Fr\"olicher sequence at $E_1$ and the existence
of sG metrics are unrelated. From the study of the sG geometry above and from Theorem~\ref{Fro-general}
we get:

\begin{theorem}\label{dim6-sG-Frolicher}
Let $M=\nilm$ be a $6$-dimensional nilmanifold endowed with an invariant complex structure $J$.
If there exists an sG metric then
the Fr\"olicher spectral sequence degenerates at the second level, i.e. $E_2(M)\cong E_{\infty}(M)$.
Moreover, if there exists an sG metric and $\frg\not\cong \frh_5$, then $E_1(M)\cong E_{\infty}(M)$.
\end{theorem}

\begin{proof}
By Proposition~\ref{dim6-sG}, the Lie algebra $\frg$ underlying $M=\nilm$ must be isomorphic to
$\frh_1,\ldots,\frh_6$ or $\frh_{19}^-$, so Theorem~\ref{Fro-general} implies
that the Fr\"olicher sequence degenerates at the second level.
The last assertion follows directly by taking into account Corollary~\ref{abelian-sG} and Table~3 below.
\end{proof}

It is interesting whether this result holds in general, that is:

\begin{question}\label{open-q}
Does the Fr\"olicher spectral sequence degenerate at the second step
for any compact complex manifold $M$ of complex dimension $3$ admitting an sG metric?
\end{question}

\medskip

In the following table we show the complex structures $J$, up to equivalence, on $\frh_1,\ldots,\frh_6$ that admit
balanced Hermitian metrics. The classification follows from the proof of Proposition~\ref{examples-sG-nobalanced}.

\bigskip

\begin{center} {\large\bf Table 3: Classification of nilpotent complex structures \\admitting balanced metrics}\vspace{5pt}

{
\begin{tabular}{|c|l|l|}
\hline
$\frg$ & \textsf{Abelian structures} & \textsf{Non-Abelian Nilpotent structures}\\
\hline && \\[-7pt]
$\frh_{1}$ &  $d\omega^2=0,\ d\omega^3=0$ &  \hskip3cm{\bf ---} \\
\hline &&\\[-9pt]
$\frh_{2}$ & \hskip1.5cm{\bf ---}
 & \hskip-.2cm$
\begin{array}{rl}
\zzz&\zzz d\omega^2 =0,\ d\omega^3 =\omega^{12} +
\omega^{1\bar1}+\omega^{1\bar2}+(x+iy)\,\omega^{2\bar2},\\[2pt]
\zzz&\zzz y>0,\quad x+y^2< \frac14
\end{array}
$ \\
\hline &&\\[-9pt]
$\frh_{3}$ & $d\omega^2 = 0,\ d\omega^3
=\omega^{1\bar{1}}- \omega^{2\bar{2}}$ & \hskip3cm{\bf ---} \\
\hline &&\\[-9pt]
$\frh_{4}$ & \hskip1.5cm{\bf ---}
& \hskip-.2cm$
\begin{array}{rl}
\zzz&\zzz d\omega^2 =0,\ d\omega^3=\omega^{12} +
\omega^{1\bar1}+\omega^{1\bar2}+x\,\omega^{2\bar2},\\[2pt]
\zzz&\zzz x<\frac14, \ x\neq 0
\end{array}
$ \\[1pt]
\hline &&\\[-9pt]
 &  & $d\omega^2 =0,\ d\omega^3=\omega^{12}$ \\
\cline{3-3}
& & \\[-8pt]
&  & $d\omega^2 =0,\ d\omega^3=\omega^{12} +
\omega^{1\bar1}+\lambda\,\omega^{1\bar 2} +(x+iy)\,\omega^{2\bar2},$ \\[2pt]
& $d\omega^2=0,$& $\mbox{with } (\lambda,x,y) \mbox{ satisfying one of:}$ \\[2pt]
$\frh_{5}$& $d \omega^3=\omega^{1\bar{1}}+\omega^{1\bar{2}}+x\,\omega^{2\bar{2}}$,& $\bullet\ \lambda=y=0, \, x\in\left(-\frac14, 0\right)$;  \\[2pt]
& $0\leq x<\frac14$& $\bullet\  0<\lambda^2<\frac 12,\ 0\leq y<\frac{\lambda^2}{2},\ x=0$;\\[3pt]
& & $\bullet\  \frac 12\leq \lambda^2< 1,\ 0\leq y<\frac{1-\lambda^2}{2},\ x=0$;\\[3pt]
& & $\bullet\  \lambda^2>1,\ 0\leq y<\frac{\lambda^2-1}{2},\ x=0$. \\ [2pt]
\hline && \\[-7pt]
$\frh_{6}$ & \hskip1.5cm{\bf ---} & $d\omega^2 =0,\ d\omega^3 =\omega^{12} + \omega^{1\bar{1}} + \omega^{1\bar{2}}$ \\
\hline
\end{tabular}}
\end{center}

\bigskip

Motivated by \cite[Theorem 1.9]{Pop2} next we study the relation between the degeneration
of the Fr\"olicher spectral sequence and the existence of sG or balanced metrics. The possibilities are
well illustrated in the following deformations of the complex structure corresponding to $\lambda=x=y=0$ on
a nilmanifold with underlying Lie algebra~$\frh_5$.

\begin{example}\label{def-on-h5}
{\rm Let us consider the Lie algebra  $\frh_5$ with the real basis $\{e^1,\ldots,e^6\}$ described in Theorem~\ref{clasif-complex}.
Let us consider the complex structure $J_{0,0}$ given by
$$
\begin{array}{lll}
J_{0,0}\,e^1 = -e^2,\quad &
J_{0,0}\,e^3 = -2e^2 - e^4,\quad & J_{0,0}\,e^5 = -e^6, \\[5pt]
J_{0,0}\,e^2 = e^1,\quad &
J_{0,0}\,e^4 = -2 e^1 +  e^3,\quad & J_{0,0}e^6 = e^5.
\end{array}
$$
With respect to the $(1,0)$-basis $\omega^k_{0,0} = e^{2k-1} - i\,J_{0,0}e^{2k-1},$ for $k=1,2,3$,
the complex structure equations are given by \eqref{epsilonzero-red-new-bis} where $(\rho, \lambda, D) = (1,0,0)$.  Therefore,
there are sG metrics (according to Remark~\ref{anysG} because $\rho=1$), there do not exist balanced metrics (see Table 3) and $E_1\not\cong E_2\cong E_{\infty}$ (by Theorem~\ref{Fro-general} (c.2) since $\rho D = 0$).


We consider the following deformation of $J_{0,0}$: 
$$
\begin{array}{ll}
J_{\lambda,0}\,e^1 = -e^2,\quad &
J_{\lambda,0}\,e^2 = e^1,\\[5pt]
J_{\lambda,0}\,e^3 = \frac{-1}{1-\lambda}(2e^2 + (1+\lambda)\,e^4),\quad &
J_{\lambda,0}\,e^4 = \frac{1}{1+\lambda}(-2 e^1 +  (1-\lambda)\,e^3),\\[5pt]
J_{\lambda,0}\,e^5 = -e^6,\quad & J_{\lambda,0}e^6 = e^5,
\end{array}
$$
where $\lambda^2\in [0,\frac12)$.  The $(1,0)$-basis $\omega^k_{\lambda,0} = e^{2k-1} - i\,J_{\lambda,0}e^{2k-1},$ for $k=1,2,3$,
satisfies  \eqref{epsilonzero-red-new-bis} where $(\rho, \lambda, D) = (1,\lambda,0)$.
If $\lambda^2\in (0,\frac12)$, then there are balanced metrics by Table 3 and $E_1\not\cong E_2\cong E_{\infty}$ by Theorem~\ref{Fro-general} (c.2).

Finally, let us consider this other deformation of $J_{0,0}$:
$$
\begin{array}{ll}
J_{0,x}\,e^1 = \frac{1}{\sqrt{1+4x}}\, \left[(4x-1)e^2 + 2x e^4\right],\quad &
J_{0,x}\,e^2 = \sqrt{1+4x}\, e^1 + \frac{2x}{\sqrt{1+4x}}\, e^3,\\[5pt]
J_{0,x}\,e^3 = -\sqrt{1+4x}\,(2e^2 + e^4),\quad &
J_{0,x}\,e^4 = -2\sqrt{1+4x}\, e^1 + \frac{1-4x}{\sqrt{1+4x}}\, e^3,\\[5pt]
J_{0,x}\,e^5 = -\sqrt{1+4x}\,e^6,\quad & J_{0,x}e^6 = \frac{1}{\sqrt{1+4x}}\,e^5,
\end{array}
$$ where $x\in(-\frac14, \infty)$.
The $(1,0)$-basis $\omega^1_{0,x} = i(e^2- i J_{0,x}e^2)$, $\omega^2_{0,x} =\frac{1}{\sqrt{1+4x}} (e^3- i J_{0,x}e^3)$, $\omega^3_{0,x} = e^5- i J_{0,x}e^5$
satisfies \eqref{epsilonzero-red-new-bis} with $(\rho, \lambda, D) = (1,0,x)$.  Using Theorem~\ref{Fro-general}, Remark~\ref{anysG} and Table 3 we get:
\begin{itemize}
\item
If $x\in (-\frac14,0)$ then there are balanced metrics and $E_1\cong E_{\infty}$.
\item
If $x\in (0,\infty)$ then there are sG metrics, there do not exist balanced metrics and $E_1\cong E_{\infty}$.
\end{itemize}

}
\end{example}

\medskip

Next we address some problems on deformation openness or closedness of several properties.
Let $\Delta$ be an open disc around the origin in $\mathbb{C}$. Following \cite[Definition 1.12]{Pop2},
a given property $\mathcal{P}$ of a compact complex manifold is said to be \emph{open}
under holomorphic deformations if for every holomorphic family of compact complex manifolds
$(M,J_a)_{a\in \Delta}$ and for every $a_0\in \Delta$ the following implication holds:

\medskip

$(M,J_{a_0})$ has property $\mathcal{P}$ $\Longrightarrow$ $(M,J_a)$ has property $\mathcal{P}$ for all $a\in\Delta$ sufficiently

\hskip4.65cm close to $a_0$.

\medskip

A given property $\mathcal{P}$ of a compact complex manifold is said to be \emph{closed} under holomorphic deformations if for every holomorphic family of compact complex manifolds $(M,J_a)_{a\in \Delta}$ and for every $a_0\in \Delta$ the following implication holds:

\medskip

$(M,J_{a})$ has property $\mathcal{P}$ for all $a\in \Delta\backslash \{a_0\}$ $\Longrightarrow$ $(M,J_{a_0})$ has property $\mathcal{P}$.

\medskip

Alessandrini and Bassanelli proved in \cite{AB} (see also \cite{FG}) that the balanced property of compact complex manifolds is not deformation open.
In contrast, Popovici has shown in \cite{Pop1} that the sG property is open under
holomorphic deformations, and conjectured in \cite[Conjectures 1.21 and 1.23]{Pop2} that both the sG and the balanced properties of compact complex manifolds are closed under holomorphic deformation.

The following result provides a counterexample to both conjectures.
For that, we start with
a nilmanifold $M$
with underlying Lie algebra isomorphic to $\frh_4$,
endowed with its abelian complex structure, which we will denote by $J_0$. We know from Corollary~\ref{abelian-sG} and Table~3 that
the complex nilmanifold $(M,J_0)$ does not admit sG metrics.
The idea is to deform holomorphically $J_0$ in an open disc $\Delta=\{ a\in \mathbb{C}\mid |a|<1 \}$ around the origin so that $J_a$ admits balanced metric for any $a\not=0$.
To find such a deformation we will use a result by Maclaughlin, Pedersen, Poon and Salamon \cite{MPPS} that describes the Kuranishi space
of the abelian complex structure $J_0$ in terms of invariant forms.
We will combine this result with our existence result
of balanced metrics (see Table~3).

\begin{theorem}\label{counterexample}
There is a holomorphic family $(M,J_a)_{a\in \Delta}$ of compact complex manifolds,
where $\Delta=\{ a\in \mathbb{C}\mid |a|<1 \}$, such that $(M,J_a)$ has balanced metrics for each $a\in \Delta\backslash \{0\}$,
but $(M,J_0)$ does not admit any strongly Gauduchon metric.  In particular, the sG property and the balanced property are not closed under holomorphic deformations.
\end{theorem}

\begin{proof}
Let $M$ be a nilmanifold with underlying Lie algebra~$\frh_4$ and let $J_0$ be its abelian complex structure.
Recall (see Table~1) that there is a (1,0)-basis $\{\omega^1,\omega^2,\omega^3\}$ for $J_0$ satisfying
$d\omega^1=d\omega^2=0$ and $d\omega^3=\omega^{1\bar{1}} +\omega^{1\bar{2}} +\frac14 \omega^{2\bar{2}}$.
However, instead of using these structure equations for $J_0$, we will consider another (1,0)-basis given by
$\{\eta^1=2\omega^1 + \omega^2, \eta^2=4i\,\omega^1 + i\,\omega^2, \eta^3=2i\,\omega^3\}$ which satisfies
$$
d\eta^1=d\eta^2=0,\quad d\eta^3=\frac{i}{2} \eta^{1\bar{1}} +\frac12 \eta^{1\bar{2}} +\frac12 \eta^{2\bar{1}}.
$$
The reason for using these complex structure equations for $J_0$ instead of the previous ones is that the latter
are better adapted to the deformation parameter space of $J_0$ found by Maclaughlin, Pedersen, Poon and Salamon in \cite{MPPS}.
In fact, using Kuranishi's method, it is proved in \cite[Example 8]{MPPS} (see also \cite{KS}) that $J_0$ has a locally complete family of deformations consisting entirely of invariant complex structures and obtained the deformation parameter space in terms of the invariant forms
$\eta^1,\eta^2,\eta^3,\eta^{\bar{1}},\eta^{\bar{2}},\eta^{\bar{3}}$.
Indeed, any complex structure sufficiently close to $J_0$ has a basis $\{\mu_\Phi^1,\mu_\Phi^2,\mu_\Phi^3\}$ of $(1,0)$-forms such that
\begin{equation}\label{def-space}
\left\{
\begin{array}{lcl}
\mu_\Phi^1 \zzz & = &\zzz \eta^1 + \Phi^1_1\, \eta^{\bar{1}} + \Phi^1_2\, \eta^{\bar{2}},\\[4pt]
\mu_\Phi^2 \zzz & = &\zzz \eta^2 +\Phi^2_1\, \eta^{\bar{1}} + \Phi^2_2\, \eta^{\bar{2}},\\[4pt]
\mu_\Phi^3 \zzz & = &\zzz \eta^3 + \Phi^3_3\, \eta^{\bar{3}},
\end{array}
\right.
\end{equation}
where the coefficients $\Phi^j_k$ are sufficiently small and satisfy the condition $i(1+\Phi^3_3)\Phi^1_2=(1-\Phi^3_3)(\Phi^1_1-\Phi^2_2)$.
Therefore, the Kuranishi space has dimension 4. Moreover, the
complex structures remain abelian if and only if $\Phi^1_2=0$ and $\Phi^1_1=\Phi^2_2$.

Next we will find a particular holomorphic deformation for $J_0$ which is not abelian and having balanced metrics.
Let $\Delta=\{ a\in \mathbb{C}\mid |a|<1 \}$. For each $a\in \Delta$, we consider the basis $\{\mu_a^1,\mu_a^2,\mu_a^3\}$ of complex 1-forms
given by
$$\mu_a^1=\eta^1 +a \eta^{\bar{1}} -i a \eta^{\bar{2}},\quad \mu_a^2=\eta^2,\quad \mu_a^3=\eta^3.$$
Note that this corresponds to take $\Phi^1_1=a$, $\Phi^1_2=-ia$ and $\Phi^2_1=\Phi^2_2=\Phi^3_3=0$ in
the parameter space~\eqref{def-space}.
Notice also that this basis defines implicitly an invariant complex structure $J_a$ on $M$ just by
declaring that the forms $\mu_a^1,\mu_a^2,\mu_a^3$ are of type (1,0) with respect to $J_a$.
Moreover, a direct calculation shows that the complex structure equations for $J_a$, with respect to this basis, are
\begin{equation}\label{ecus-disc-a}
d\mu_a^1=d\mu_a^2=0,\quad 2(1-|a|^2) d\mu_a^3=2\bar{a} \mu_a^{12}+i \mu_a^{1\bar{1}} +\mu_a^{1\bar{2}} + \mu_a^{2\bar{1}} -i|a|^2 \mu_a^{2\bar{2}},
\end{equation}
for each $a\in\Delta$.

Recall that by Corollary~\ref{abelian-sG} and Table~3, if $a=0$ then the complex nilmanifold $(M,J_0)$ does not admit sG metrics
because $J_0$ is abelian.

For each $a\in \Delta\backslash \{0\}$ the complex structure is nilpotent but not abelian. In this case we can normalize the coefficient
of $\mu_a^{12}$ by taking $\frac{1-|a|^2}{\bar{a}}\mu_a^3$ instead of $\mu_a^3$, so we can suppose that the complex structure equations are
$$
d\mu_a^1=d\mu_a^2=0,\quad d\mu_a^3= \mu_a^{12}+\frac{i}{2\bar{a}} \mu_a^{1\bar{1}} + \frac{1}{2\bar{a}}(\mu_a^{1\bar{2}} + \mu_a^{2\bar{1}}) -\frac{ia}{2} \mu_a^{2\bar{2}}.
$$
Moreover, with respect to the (1,0)-basis given by $\{\tau_a^1=\mu_a^1-i\mu_a^2$, $\tau_a^2=-2\bar{a}i\,\mu_a^2$, $\tau_a^3=-2\bar{a}i\,\mu_a^3\}$,
the structure equations for $J_a$ become
$$
d\tau_a^1=d\tau_a^2=0,\quad d\tau_a^3=\tau_a^{12} + \tau_a^{1\bar1} -\frac{1}{a}\, \tau_a^{1\bar2} + \frac{1-|a|^2}{4|a|^2}\,\tau_a^{2\bar2}.
$$
Now, according to Proposition~\ref{nueva} we can suppose that the coefficient of $\tau_a^{1\bar2}$ is equal to $1/|a|$
(in fact, any complex structure given by $(\rho, B, D)$ is equivalent to $(\rho, |B|, D)$ where $B\in\mathbb C$).

In conclusion, for each $a\in \Delta\backslash \{0\}$ there exists a (1,0)-basis for which the complex equations
are of the form \eqref{J-nilp} with $\lambda=\frac{1}{|a|}$ and $D=\frac{1-|a|^2}{4|a|^2}$.
Taking $x=\Real D=\frac{1-|a|^2}{4|a|^2}$ and $y=\Imag D=0$, one has $\lambda^2-4x=1$.
Now, following the proof of Proposition~\ref{examples-sG-nobalanced}, since $\lambda\not=0$ the complex structure $J_a$
admits a balanced metric if and only if
\eqref{lambda-balanced} is satisfied. But the latter condition reads
$$
\lambda^2(\lambda^2-4x)=\frac{1}{|a|^2}>0,
$$
so there exists a balanced Hermitian metric for each $a\in \mathbb{C}$ such that $0<|a|<1$.
\end{proof}

\begin{remark}\label{failure}
{\rm
It is worth giving a closer look at the failure of the sG property at $a=0$.
Let $J_a$ be the family of complex structures given by
\eqref{ecus-disc-a} for any $a\in\Delta=\{a\in \mathbb{C}\mid |a|<1 \}$, and
let us consider the real 2-form $\Omega$
of bidegree (1,1) for $J_a$
given by
$$
2\Omega=i r^2\, \mu_a^{1\bar{1}} + i s^2\, \mu_a^{2\bar{2}} + i t^2\, \mu_a^{3\bar{3}},
$$
where $r,s,t\in \mathbb{R}$. Since
$$
4\Omega\wedge d\Omega=\frac{it^2}{2(1-|a|^2)}(s^2-|a|^2r^2)(\mu_a^{12\bar{1}\bar{2}\bar{3}}-\mu_a^{123\bar{1}\bar{2}}),
$$
the 4-form $\Omega^2$ is closed if and only if $s^2=|a|^2r^2$, i.e. if and only if $\Omega$ is given by
$$
2\Omega=i r^2\, \mu_a^{1\bar{1}} + i |a|^2r^2\, \mu_a^{2\bar{2}} + i t^2\, \mu_a^{3\bar{3}}.
$$
This defines a balanced $J_a$-Hermitian metric for any $r,t>0$ and for any $0<|a|<1$.
However, in the central limit, $a=0$ and the form $\Omega$ becomes degenerate, i.e. $\Omega^3=0$,
therefore it does not define a $J_0$-Hermitian metric because the fundamental form of any Hermitian structure
is always non degenerate.
}
\end{remark}

It is well known that the property of \emph{``the Fr\"olicher spectral sequence degenerating at $E_1$''}
is open under holomorphic deformations. In \cite[Theorem 5.4]{ES} it is proved that this property
is not closed under holomorphic deformations. As a consequence of
Theorems~\ref{Fro-general} and~\ref{counterexample}
we obtain another example based on the complex geometry of $\frh_4$.

\begin{corollary}\label{nonclosedFro}
Let $(M,J_0)$ be a nilmanifold with underlying Lie algebra~$\frh_4$ endowed with abelian complex structure $J_0$.
There is a holomorphic family of compact complex manifolds $(M,J_a)_{a\in \Delta}$,
where $\Delta=\{a\in \mathbb{C}\mid |a|<1 \}$, such that $E_1(M,J_a)\cong E_{\infty}(M,J_a)$ for each $a\in \Delta\backslash \{0\}$,
but $E_1(M,J_0)\not\cong E_{\infty}(M,J_0)$.
\end{corollary}

The upper semicontinuity of the Hodge numbers is crucial in the proof of the openness of the property ``$E_1\cong E_{\infty}$''.
Since we proved in Corollary~\ref{no-semi-continuous} that the upper semicontinuity fails for $E_2^{p,q}$, the following result
is not so unexpected.

\begin{corollary}\label{nonopenE2}
The property of \emph{``the Fr\"olicher spectral sequence degenerating at $E_2$''}
is not open.
\end{corollary}

\begin{proof}
The family $J_t$ given in Example~\ref{example-h15} satisfies
$E_2(J_{-\frac{\pi}{2}})\cong E_{\infty}(J_{-\frac{\pi}{2}})$, because $J_{-\frac{\pi}{2}}$ is in case (f.1) of Theorem~\ref{Fro-general},
but $E_2(J_{t})\not\cong E_{\infty}(J_{t})$
for $t\in(-\frac{\pi}{2},\frac{\pi}{2})$.
\end{proof}

This result is relevant in relation to Question~\ref{open-q} since the existence of sG metric is an open property.
Notice that there is no contradiction because the Lie algebra in Example~\ref{example-h15} is $\frh_{15}$ which does not admit any sG metric.

\bigskip

\noindent {\bf Acknowledgments.}
We would like to thank D. Popovici for very useful comments concerning deformation stability of the sG and the balanced properties.
At the time of writing the version
of June 2012 including Theorem~\ref{counterexample},
we learned that M. Verbitsky may also have found independently a counterexample to the closedness conjecture for the
sG property.
This work has been partially
supported through Projects MICINN (Spain) MTM2008-06540-C02-02, MTM2010-19336 and MTM2011-28326-C02-01.
Raquel Villacampa would like to thank the IMUS (Instituto de Matem\'aticas de la Universidad de Sevilla Antonio de Castro Brzezicki) for having
awarded her with a post-doctoral stay in Seville.
Finally, we also thank the referee for many useful comments and suggestions that have helped us to improve the final version of the paper.

\end{document}